\documentclass[reqno]{amsart}
\usepackage{amsmath, amssymb, amsfonts, amsthm}
\usepackage[english]{babel}
\usepackage{hyperref}
\addto\extrasenglish{%
}
\usepackage{xcolor}
\usepackage{graphicx}
\usepackage{cleveref}
\usepackage[utf8]{inputenc}

\usepackage[margin=1.15in]{geometry}

\usepackage[normalem]{ulem}

\makeatletter
\newtheorem*{rep@theorem}{\rep@title}
\newcommand{\newreptheorem}[2]{%
\newenvironment{rep#1}[1]{%
 \def\rep@title{#2 \ref{##1}}%
 \begin{rep@theorem}}%
 {\end{rep@theorem}}}
\makeatother

\theoremstyle{plain}
\newtheorem{theorem}{Theorem}[section] % reset theorem numbering for each chapter
\newtheorem{corollary}[theorem]{Corollary}
\newreptheorem{theorem}{Corollary}
\newtheorem{lemma}[theorem]{Lemma}
\newtheorem{proposition}[theorem]{Proposition}

\theoremstyle{remark}
\newtheorem{remark}[theorem]{Remark}

\usepackage{bm}

\newcommand{\prob}[1]{\mathbb{P}\left[#1\right]}
\newcommand{\E}[1]{\mathbb{E}\left[#1\right]}
\newcommand{\limn}[1]{\lim\limits_{n\rightarrow\infty}}

\def\Im{\mathop{\mathrm{Im}}\nolimits}

\def\Z{\mathrm{Z}}
\def\N{\mathbb{N}}
\def\H{\mathbb{H}}

\def\dist{\mathrm{dist}}
\def\diam{\mathrm{diam}}
\def\cle{\mathrm{CLE}}

\def\ls{\mathcal{L}}
\def\bls{\mathcal{B}}

\def\cro{\mathrm{Cross}}
\def\clu{\mathrm{Clus}}
\def\com{\mathrm{Comp}}
\def\lm{\mu^{\text{loop}}}

\usepackage{mathrsfs, stmaryrd, mathtools}

\DeclarePairedDelimiterX\intff[2]{[}{]}{#1,#2}
\DeclarePairedDelimiterX\intfo[2]{[}{)}{#1,#2}
\DeclarePairedDelimiterX\intof[2]{(}{]}{#1,#2}
\DeclarePairedDelimiterX\intoo[2]{(}{)}{#1,#2}

\DeclarePairedDelimiterX{\setof}[2]{\lbrace}{\rbrace}{#1\,{:}\,#2}
\DeclarePairedDelimiterX{\bracksof}[2]{[}{]}{#1\,\delimsize\vert\,#2}
\DeclarePairedDelimiterX{\parsof}[2]{(}{)}{#1\,\delimsize\vert\,#2}
\DeclarePairedDelimiterXPP\lnorm[2]{}\lVert\rVert{_{#1}}{#2}

\begin{document}

% Enter full title and short title for running headers
\title{On the crossing estimates for simple conformal loop ensembles}
\maketitle

\begin{center}
\author{Tianyi Bai\footnote{NYU-ECNU Institute of Mathematical Sciences at NYU Shanghai, China.} 
\and 
Yijun Wan\footnote{Département de Mathématiques et Applications, École Normale Supérieure PSL Research University,
CNRS UMR 8553, 45 rue d’Ulm, Paris, France.}}
\end{center}

\begin{abstract}
We prove a super-exponential decay of probabilities that there exist $n$ crossings of a given quad for a simple $\cle_\kappa(\Omega)$, $\frac{8}{3}<\kappa\le 4$, as $n$ goes to infinity. Besides, being of independent interest, this also provides the missing ingredient in \cite{basok-chelkak} for proving the convergence of probabilities of cylindrical events for the double-dimer loop ensembles to those for the nested $\cle_4(\Omega)$.
\end{abstract}

\section{Introduction}
In statistical physics, crossing-type estimates or regularity properties of the continuum limiting objects can be instrumental to study the scaling limits of certain models. In our paper, we are interested in the crossing numbers of simple conformal loop ensembles $\cle_\kappa$, $\frac{8}{3}<\kappa\le 4$. Let $\Omega$ be a simply connected subdomain of the upper half-plane $\H$ and $\cle_\kappa(\Omega)$ be a non-nested simple conformal loop ensemble in $\Omega$ with $\frac{8}{3}<\kappa\le 4$. The main result of the present paper is on the super-exponential decay, as $n\to\infty$, of the probability of finding $n$ crossings of a fixed annulus $A(r,R)$ ($\Omega\cap A(r,R)$ needs not to be connected, see Figure \ref{pic:1}) or of a fixed quad $Q$ with two opposite sides attached to $\partial\Omega$ for $\cle_\kappa(\Omega)$ (see Figure \ref{pic:quad}).

\begin{figure}[ht]
\centering
\includegraphics[width=0.5\textwidth]{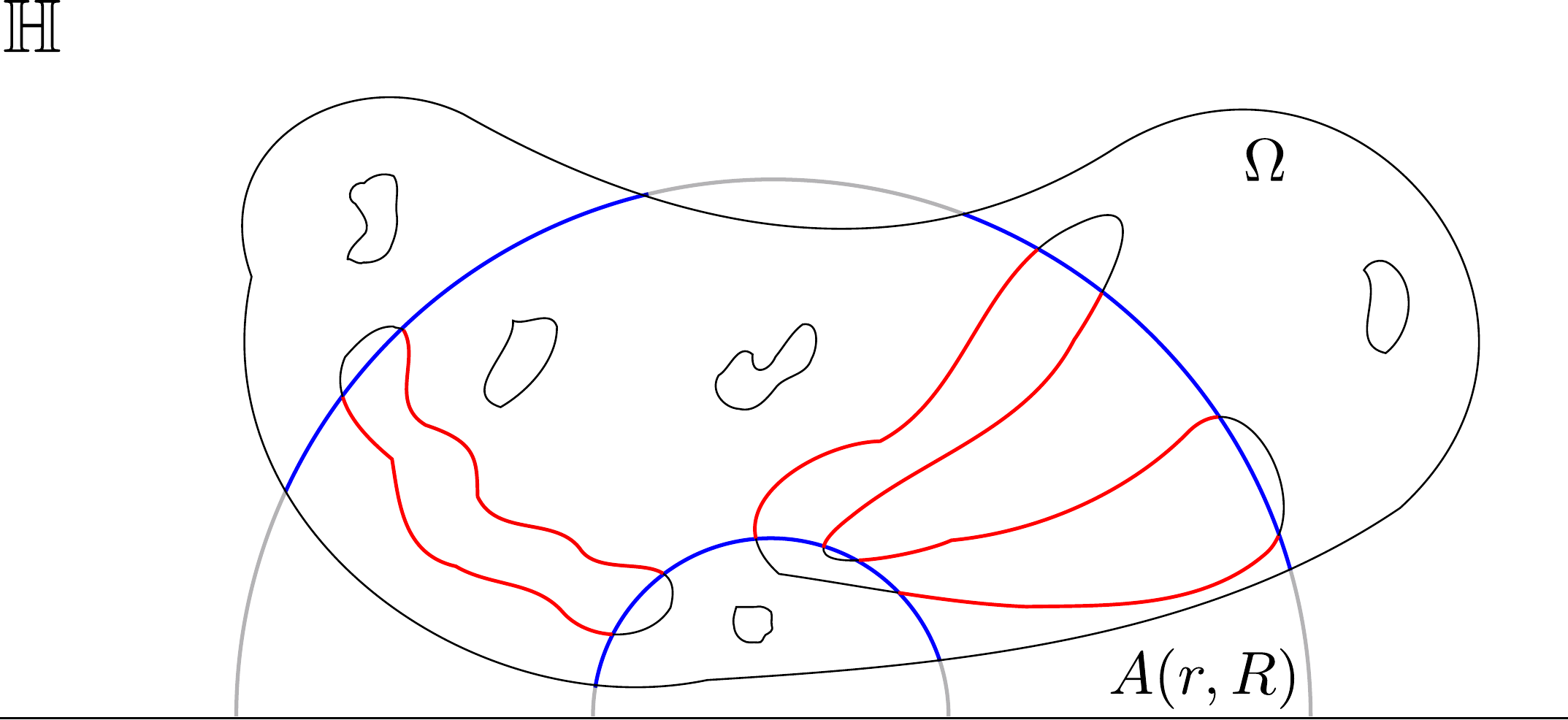}
\caption{
In this illustration, crossings connect the opposite blue arcs of $\partial A(r,R)$, and we have $6$ crossings given by the red paths.}
\label{pic:1}
\end{figure}

\subsection{Background on CLEs}

Conformal loop ensemble, $\cle_\kappa$ for $\frac{8}{3}<\kappa< 8$, is a random countable collection of loops in a (simply connected) planar domain $\Omega\neq\mathbb{C}$, which can be viewed as the full-picture version of the Schramm-Loewner evolution (SLE). It was introduced by Sheffield in \cite{sheffield} as candidates for the scaling limits of certain statistical physics models at critical temperature, which can be interpreted as random collections of disjoint, non-self-crossing loops. $\cle_\kappa$ is shown to be the scaling limit of : critical Ising model $\kappa=3$ \cite{benoist-hongler}, FK-Ising percolation $\kappa=16/3$ \cite{kemppainen-smirnovII}, and critical site percolation on the triangular lattice $\kappa=6$ \cite{camia-newman}. Beyond these, $\cle_\kappa, \frac{8}{3}<\kappa\le4$, is conjectured to describe the scaling limit of the loop $O(n)$ model if $n=-2\cos(4\pi/\kappa)\in(0,2]$ while $\cle_{\kappa}, 4<\kappa<8$, is conjectured to be the scaling limit of the FK($q$)-percolation model if $q= 4\cos^2(4\pi/\kappa)$.

CLE is characterized by a parameter $\kappa\in(8/3,8)$, describing the density of loops in it. All loops of a sample of $\cle_\kappa$ are simple, do not intersect each other, and do not intersect the domain boundary when $\kappa\in(\frac{8}{3},4]$.  When $\kappa\in(4,8)$, the loops are self-intersecting (but not self-crossing) and may touch (but not cross) other loops and the domain boundary.

Since such critical models are expected to be conformally invariant on large distance scales, CLEs are defined to be conformally invariant: if $\varphi: \Omega \rightarrow \Omega'$ is a conformal map and $\Gamma$ is a $\cle_\kappa$ in $\Omega$, then $\varphi(\Gamma)$ is a $\cle_\kappa$ in $\Omega'$.

For each $\kappa$, there are two versions of CLEs: \emph{non-nested} and \emph{nested}, the latter is obtained from the former by recursively iterating the construction inside each loop constructed in the previous step. In this article, we are mainly interested in the non-nested $\cle_\kappa$ for $\kappa\in(\frac{8}{3},4]$ (except for Section \ref{sec:complexity}, where we consider nested CLEs). CLEs can be constructed using one of the two natural conformally invariant probability measures on curves, the Brownian motion (BM) and the Schramm-Loewner evolution (SLE). The BM-based construction is the main tool that we will use in this paper, see Section \ref{sec:Brownian loop measure} and Section \ref{sec:loop-soup} below.
In this approach, the non-nested simple $\cle_\kappa$, $\frac{8}{3}<\kappa\le 4$, is obtained as the collection of outermost boundaries of clusters appearing in a Poisson process of Brownian loops. It is worth noting that this construction admits a discretization: the scaling limit of outer boundaries of outermost clusters of random walk loop-soup is a CLE. Such convergence was first discovered in \cite{brug-camia-lis}, focusing on the scaling limit of outermost boundaries of clusters of loops with some microscopic loops neglected. Then the convergence of outermost boundaries of clusters of the full loop ensemble was proved in \cite{lupu}, by considering the special case $\kappa=4$, using its connection to the Gaussian free field (GFF): $\cle_4$ loops are the “level lines” of the GFF \cite{qian-werner}.

\subsection{Super-exponential decaying crossing estimates for non-nested simple $\cle$}
\label{sec:1.2}

The main result in this paper is the following.
\begin{theorem}\label{thm:main}
Given a simply connected domain $\Omega$ of the upper half-plane $\H=\setof{z\in\mathbb C}{\Im z> 0}$, let $\cle_{\kappa}(\Omega)$ be a \emph{non-nested} simple conformal loop ensemble with $\kappa\in(\frac{8}{3},4]$ in $\Omega$.
Let $0<r<R$, denote by $\cro_{A(r,R)}(\cle_\kappa(\Omega))$ the number of disjoint arcs in $\cle_\kappa(\Omega)$ joining the two boundaries of $A(r,R):=\setof{z\in\mathbb C}{r<|z|<R}$, see Figure \ref{pic:1}. Then for any $s>0$, 
\[
\sup_{\Omega\subseteq\mathbb H}\mathbb{P}[\cro_{A(r,R)}(\cle_\kappa(\Omega))\ge n]= O(s^n),
\]
where the supremum is taken over all simply connected domains $\Omega\subseteq\mathbb H$, and the constant in $O(s^n)$ depends on $\kappa$, $s$ and $R/r$.
\end{theorem}

\begin{remark}\label{rmk:thm_main} 
\begin{itemize}
\item 
By the BK's inequality \cite{berg}, it is not hard to see that $\mathbb{P}[\cro_{A(r,R)}(\cle_\kappa(\Omega))\ge n]$ decays at least exponentially fast, see e.g. \cite[Lemma 9.6]{sheffield-werner}.
%For the crossing estimates of $\cle_\kappa$, $\frac{8}{3}<\kappa\le 4$, similar exponential decay for the probability of more than $n$ crossings has appeared in \cite[Lemma 9.6]{sheffield-werner}.

\item  The domain Markov property of CLEs requires conditioning on entire loops, from which we can only obtain a super-exponential decay of probabilities on the cluster number (of a Brownian loop soup) defined in Section \ref{sec:def_number}, see Proposition \ref{prop:clusterMajor}. Nevertheless, Theorem \ref{thm:main} can be deduced from Proposition \ref{prop:clusterMajor} by our estimates of the component number, see Lemma \ref{lm:CompToCluster} and Proposition \ref{prop:loop_crossing}.

\item The arm exponents for SLE discussed in \cite{wu-zhan} cannot be applied in our circumstance since the asymptotic regime in \cite{wu-zhan} is different, sending
$\frac{R}{r}\rightarrow\infty$ rather than $n\rightarrow\infty$. Using certain martingales for SLEs and the conformal domain Markov property, the methods in \cite{wu-zhan} involves distortion when conformally mapping the slit domain to the half-plane during each iteration, which gives rise to a super-exponential growing multiplicative factor for the crossing estimates of a fixed quad as the number of crossings goes to infinity.
\item 
We conjecture that the analogue of Theorem \ref{thm:main} for nested CLEs is valid as well.
%It would be interesting to obtain an analogue of Theorem \ref{thm:main} for nested CLEs as well. 
However, our argument fails in that case because nested CLEs cannot be constructed from a single Brownian loop-soup; besides, the estimates in Theorem \ref{thm:main} are not enough to deduce that the total crossing number resulted from the branching structure of nested CLEs has super-exponential decay (the expectation of the crossing number for a simple non-nested CLE may be simply larger than one, resulting in a supercritical branching process).
\end{itemize}
\end{remark}

Though the result of Theorem \ref{thm:main} does not yet have applications to the convergence of loop representations of statistical physics models other than double-dimers to $\cle_4$ (see Section \ref{sec:1.5}), it could be used in the same vein if a relevant topological observables framework is developed for $\kappa\le 4$. It would be also interesting to study similar crossing estimates in the case $\kappa>4$, which probably should rely upon the branching $\text{SLE}_\kappa$ techniques instead of the Brownian loop-soup. See also \cite{hongler} for a study on similar crossing events of the critical site percolation on the triangular lattice, whose scaling limit is known to be the nested $\cle_6$.

%Let us emphasize that taking the supremum over all subdomains $\Omega\subseteq\H$ in Theorem \ref{thm:main} is crucial for the proof of Corollary \ref{cor:complexity}. However, this result is of independent interest even for a single domain $\Omega\subseteq\H$. 

Moreover, one can extend the result for the crossing event of $A(r,R)$ to more general quads on any proper domain $\Omega\subset\mathbb C$. We define a \emph{crossing-quad} of $\Omega$, denoted by $Q=(V;S_k,k=0,1,2,3)$, to be a simply connected subset $V$ inside $\Omega$, whose boundary consists of four arcs $S_{k},\ k=0,1,2,3$ listed in the counterclockwise order, such that $S_1,S_3\subset\partial \Omega$ (see e.g. Figure \ref{pic:quad}).
A natural conformally invariant measurement of the width of a quad $Q$ is  the \emph{conformal modulus} $m(Q)$, defined as the unique number for which $Q$ can be conformally mapped onto a rectangle $[0,1]\times[0,m(Q)]$, such that $S_k$ are mapped to the four sides of the rectangle with $S_0$ mapped to $[0,1]\times\{0\}$. We refer interested readers to \cite{ahlfors} for more details about properties of these concepts.

\begin{figure}[ht]
\centering
\includegraphics[width=0.3\textwidth]{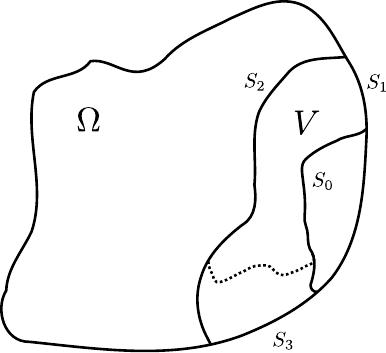}
\caption{
A crossing-quad $(V;S_0,S_1,S_2,S_3)$ in $\Omega$ and an (dotted) arc crossing it.}
\label{pic:quad}
\end{figure}

We can deduce from Theorem \ref{thm:main} that
\begin{corollary}\label{cor:main2}
Let $Q=(V;S_k,k=0,1,2,3)$ be a crossing-quad in a proper subdomain $\Omega$ of $\mathbb C$, and denote by $\cro_Q(\cle_\kappa(\Omega))$ the number of (disjoint) arcs in $\cle_\kappa(\Omega)$ joining $S_0$ and $S_2$ inside $V$. Then for any $s>0$ and $m_0>0$, 
\[
\sup_{\Omega\subseteq\H}\mathbb{P}[\cro_Q(\cle_\kappa(\Omega))\ge n]= O(s^n)
\]
uniformly over the quad $Q$ such that $m(Q)>m_0$, where the constant in $O(s^n)$ depends on $\kappa$, $m_0$ and $s$.
\end{corollary}

\begin{remark}
The proof of Corollary \ref{cor:main2} (see Section \ref{sec:general_ann}) also applies to estimating the crossing number of an annulus contained inside the domain, see e.g. Lemma \ref{lm:cro_annuli}.
\end{remark}

\subsection{Convergence of probabilities of cylindrical events for double-dimers}\label{sec:1.5}

Besides studying regularity properties of $\cle_\kappa(\Omega)$, this paper also serves as a complement to the papers \cite{basok-chelkak} and \cite{dubedat} regarding the convergence of double-dimer loop ensembles to $\cle_4$. Developing the ideas of Kenyon \cite{kenyon}, Dubédat proved the convergence of the so-called \emph{topological observables} of double-dimer loop ensembles in Temperleyan domains to an appropriately defined Jimbo-Miwa-Ueno isomonodromic tau function, see \cite{dubedat}. Later on, based on an analysis of expansions of entire functions (defined on the moduli space of $\text{SL}_2(\mathbb{C})$-representations of the fundamental group of a punctured domain) with respect to the Fock-Goncharov \emph{lamination basis}, Basok and Chelkak \cite{basok-chelkak} proved the convergence of probabilities of cylindrical events for the double-dimer loop ensemble to the coefficients of the (infinite series) expansion of the isomonodromic tau-function via the lamination basis. On the other hand, it was shown by Dubédat \cite[Theorem 1]{dubedat} that this tau-function can be obtained by taking the expectation of the product of the traces of loop monodromies over $\cle_4$ provided that the monodromy is close enough to the identity. By definition, this provides another expansion of the tau-function via the lamination basis, where the coefficients are equal to probabilities of cylindrical events evaluated for $\cle_4$. It follows from \cite[Theorem 1.4]{basok-chelkak} that the equality of two infinite series expansions via the lamination basis implies the equality of their coefficients, provided that the coefficients of both expansions decay super-exponentially. Therefore, if one knows that the probabilities of cylindrical events evaluated for $\cle_4$ decay super-exponentially, then the results of \cite{basok-chelkak} and \cite{dubedat} imply the convergence of probabilities for the double-dimer loop ensembles to those of $\cle_4$, see \cite[Corollary 1.7]{basok-chelkak} and Corollary \ref{cor:complexity} below.

%As an application, Theorem \ref{thm:main} confirms the conjecture in \cite[Corollary 1.7]{basok-chelkak}.
%with the assumption therein fixed, {\red which is a result on the convergence of probabilities of cylindrical events  for the double-dimer model, see \eqref{eq:cv_dbl-dimer}. While we do not use the definition of the double-dimer model in this paper, let us introduce it briefly for the sake of completeness. 
Given a Temperlean simply connected approximation $\Omega^\delta\subseteq \delta\Z^2$ of $\Omega$, the double-dimer loop ensemble on $\Omega^\delta$ is obtained by superimposing two independent uniform dimer configurations on $\Omega^\delta$. This produces a number of
loops and double-edges, with the latter withdrawn. Obtained in this manner, we denote by $\Theta^\delta_\Omega$ the random collection of nested simple pairwise disjoint loops  on $\Omega^\delta$.

%To state Corollary \ref{cor:complexity}, we give a concise definition of the complexity of a loop ensemble here.
Given a collection of pairwise distinct punctures 
{in a simply connected domain}, 
$\lambda_1,\ldots,\lambda_N\in\Omega$, a \emph{macroscopic lamination} on $\Omega\setminus\{\lambda_1,\ldots,\lambda_N\}$ is a {finite} collection of disjoint simple loops surrounding at least two punctures considered up to homotopies. We fix once and forever a triangulation of $\Omega\setminus\{\lambda_1,\ldots,\lambda_N\}$ with vertices at $\lambda_1,\ldots,\lambda_N$, $\partial\Omega$ (see \cite{basok-chelkak} for more details) and define the \emph{complexity} $|\Gamma|$ of a lamination $\Gamma$ to be the number of intersections of loops in $\Gamma$ with the edges of the triangulation (computed after resolving all unnecessary intersections). Note that the complexity $\Gamma$ cannot be estimated via the number of loops in $\Gamma$ if $N\ge 3$: one can construct a lamination consisting of one loop but having arbitrary large complexity, see Figure \ref{pic:3}.

\begin{figure}[ht]
\centering
\includegraphics[width=0.5\textwidth]{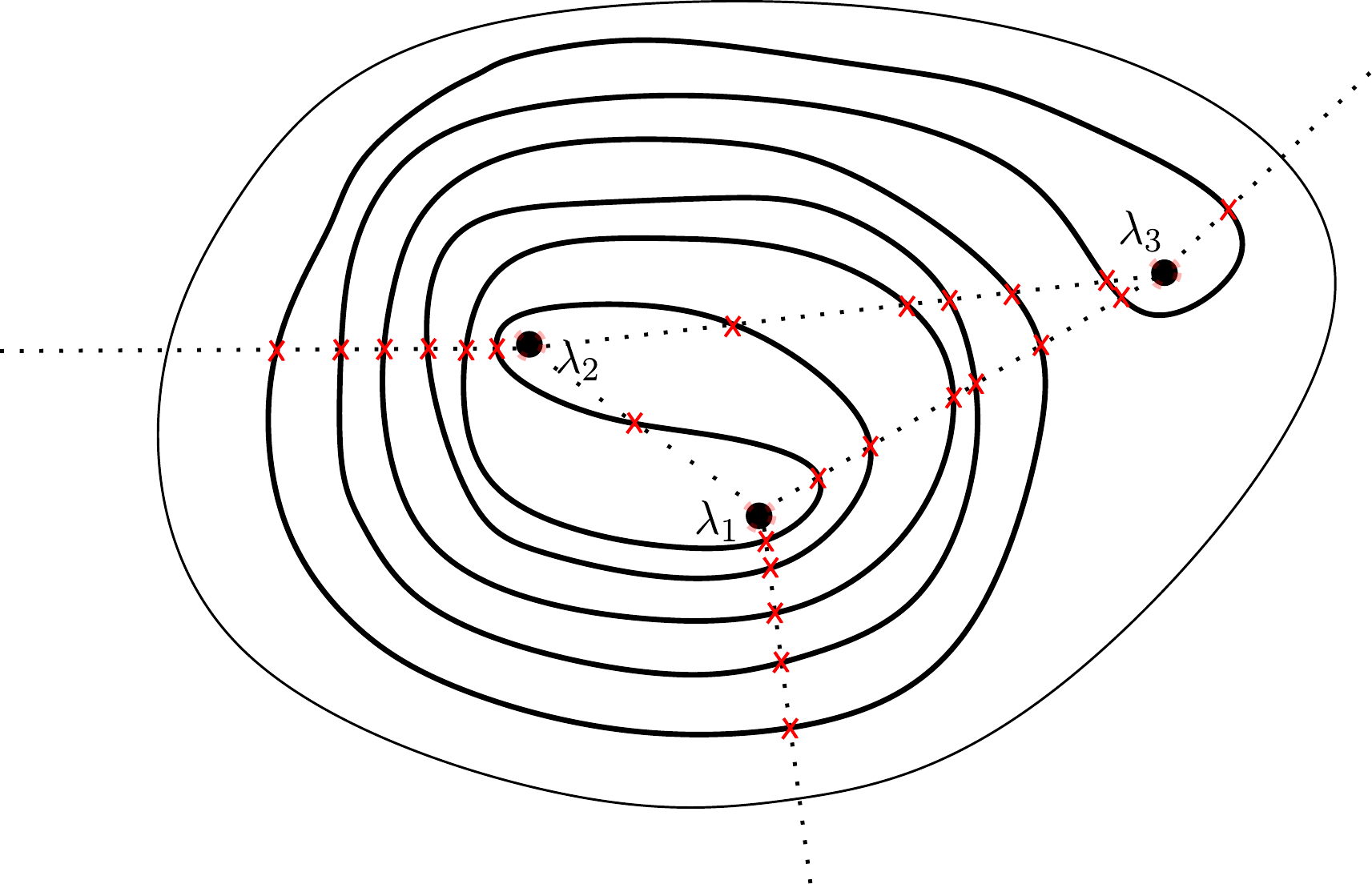}
\caption{
One loop (the bold one) with complexity $24$ (the minimal number of crosses of a loop within this homotopy class with the edges of the triangulation).}
\label{pic:3}
\end{figure}

\begin{corollary}[Convergence of probabilities of cylindrical events of double-dimer configuration] \label{cor:complexity}
Let $\Theta_\Omega$ be a nested $\cle_\kappa$ in a simply connected domain $\Omega$, $\kappa\in(\frac{8}{3},4]$. 
%whose definition is given in the following subsection
Let $\Gamma$ be a macroscopic lamination, and denote by $\Theta_\Omega \sim \Gamma$ the event that  $\Theta_\Omega$ is homotopic to $\Gamma$ after one removes all loops surrounding at most one puncture. Then for any $s>0$,
	\[
\mathbb{P}_{\cle^\text{nested}_\kappa}[\Theta_\Omega\sim\Gamma]=O(s^{-|\Gamma|}) \text{  as }|\Gamma| \rightarrow\infty.
\]
Furthermore, for all macroscopic laminations $\Gamma$, 
\begin{equation}\label{eq:cv_dbl-dimer}
\mathbb{P}_{\text{double-dimer}}[\Theta^\delta_\Omega\sim\Gamma]\to\mathbb{P}_{\cle^\text{nested}_4}[\Theta_\Omega\sim\Gamma] \text{  as } \delta\to 0.
\end{equation}
\end{corollary}

It is worth mentioning that the estimate provided in Corollary \ref{cor:complexity} is weaker than the super-exponential decay of crossing numbers of nested CLEs. However, it is sufficient for the analysis performed in \cite{basok-chelkak}. 

The rest of the paper is organized as follows: 
Section \ref{sec:defs} discusses several quantities related to the crossing number, presents the Brownian loop-soup construction of CLEs, and gives a proof outline for our main results. Section \ref{sec:comp_clus_number} is around some preliminary deterministic results and the technical proof of Proposition \ref{prop:clusterMajor}. The readers not interested in these details may skip Section \ref{sec:comp_clus_number}. In the end, the proofs of Theorem \ref{thm:main}, Corollary \ref{cor:main2} and Corollary \ref{cor:complexity} are given in Section \ref{sec:main_proof}, Section \ref{sec:general_ann} and Section \ref{sec:complexity} respectively.

\section{Notations and Preliminaries}\label{sec:defs}
In this section, we fix and discuss some notations for loop ensembles and introduce the Brownian loop-soup construction of the simple CLE.

\subsection{Clusters, crossing and component number}\label{sec:def_number}
Given a simply connected domain $\Omega$, a \emph{loop ensemble} $\ls$ in $\Omega$ is a countable collection of loops (not necessarily simple or pairwise disjoint) in $\Omega$. Two loops $l$ and $l'$ are in the same cluster if and only if one can find a finite chain of loops $l_0,\ldots,l_n$ in $\ls$ such that $l_0=l,\ l_n=l'$ and $l_j\cap l_{j-1}\neq\emptyset$ for all $j=1,\ldots,n$. Given a cluster $C$, we denote by $\overline C$ the closure of the union of all loops in $C$. Denote by $F(C)$ the filling of $C$, which is the complement of the unbounded connected component of $\mathbb{C}\setminus\overline{C}$. ({Note that F(C)} is simply connected). A cluster $C$ is called outermost is there exists no cluster $C'$ such that $C\subset F(C')$. Denote by $F(\ls)$ the family $\setof{F(C)}{C\text{ is a outermost cluster of }\ls}$.  

A loop ensemble $\ls$ can be divided into two parts by restriction to a subdomain $\Omega'\subset\Omega$,
\[
\ls(\Omega'):=\setof{l\in\ls}{l\subset\Omega'},\,\ls(\Omega')^\bot:=\ls\backslash\ls(\Omega'),
\]
One can also divide $\ls$ by considering the loop diameter:
\[
\ls_{<a}:=\setof{l\in\ls}{\diam(l)<a},\,\ls_{\ge a}:=\setof{l\in\ls}{\diam(l)\ge a},
\]
where $\diam(l):=\sup_{x,y\in l}\dist(x,y)$.

For all $0<r<R$ and point $z_0\in\mathbb C$, denote by $A_{z_0}(r,R)$ the annulus of inner and outer radii $r$ and $R$
centered at $z_0$,

\begin{equation}\label{eq:annulus}
A_{z_0}(r,R) = \{z \in \mathbb C, r < |z-z_0| < R\},
\end{equation}

and denote by $B_r(z_0)$ the disk of radius $r$ centered at $z_0$,
\begin{equation}\label{eq:C}
B_r(z_0)=\{z\in\mathbb C, |z-z_0|\le r\}.
\end{equation}

For the sake of simplicity, we will write $A(r,R)$ and $B_r$ if $z_0$ is the origin $0$ of the complex plane.

We say that a connected set crosses $A=A(r,R)$, if it intersects both boundaries of $\partial A(r,R)=\partial B_R\sqcup\partial B_r$.
For a loop ensemble $\ls$, the \emph{crossing number} $\cro_A(\ls)$ is the maximum number of disjoint arcs of loops in $\ls$ that cross $A$, see Figure \ref{pic:cross}. From the definition, one can easily observe that the crossing number is monotone and subadditive, that is
\begin{equation}\label{eq:property_cross}
\cro_A(\mathcal L_1)\le\cro_A(\mathcal L_1\cup\mathcal L_2)\le\cro_A(\mathcal L_1)+\cro_A(\mathcal L_2).
\end{equation}

\begin{figure}[ht]
\centering
\includegraphics[width=0.5\textwidth]{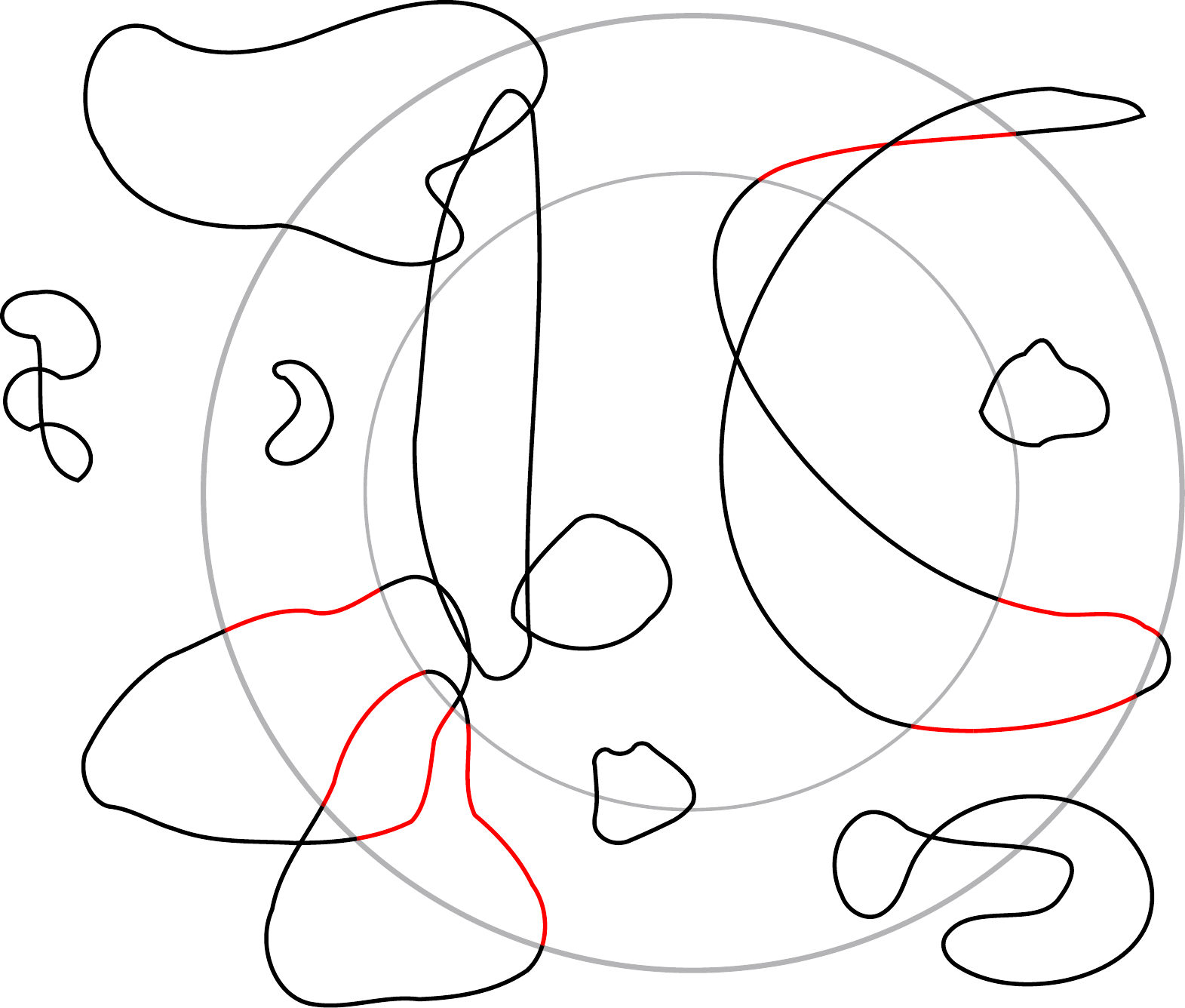}
\caption{
In this configuration, we have $\cro_A(\mathcal L) = 7$ with the 7 paths marked in red. Notice that one can choose only one among the two paths in the top-right part because they intersect each other. Since loops in CLE do not self-intersect, this will never happen for a CLE.}
\label{pic:cross}
\end{figure}

The \emph{component number $\com_A(\ls)$} is defined as the number of path-connected components of $\cup_{C\in\{\text{outermost clusters of }\ls\}}F(C)\cap A$ that cross $A$.
In particular, if $\ls$ is a non-nested simple loop ensemble with disjoint loops, for instance the non-nested $\cle_\kappa, \frac{8}{3}<\kappa\le4$, then
\begin{equation}\label{eq:cro-comp}
\cro_A(\ls)=2\com_A(\ls).
\end{equation}
\begin{figure}[ht]
\centering
\includegraphics[width=0.5\textwidth]{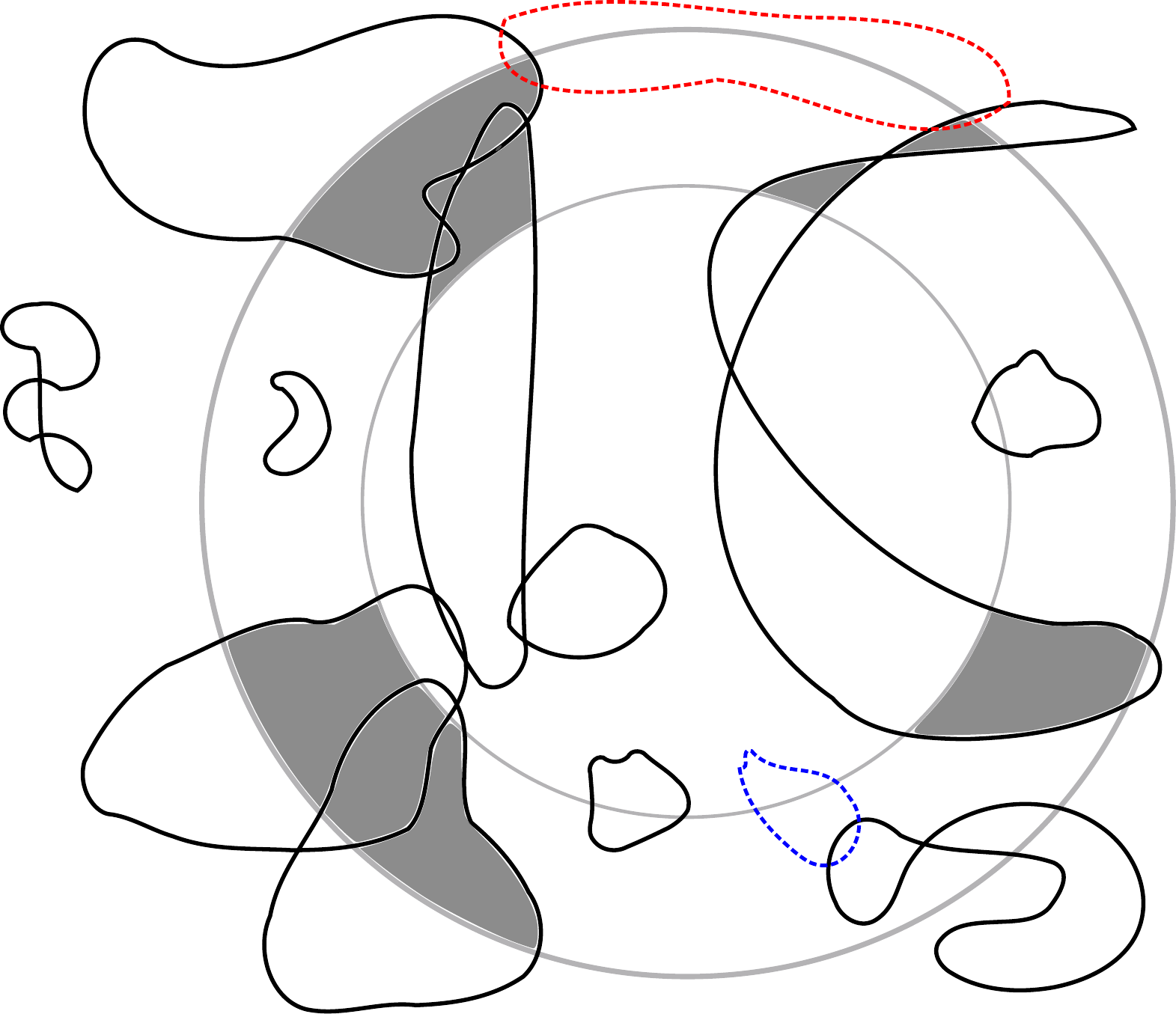}
\caption{
We have 4 crossing components marked in gray, therefore $\com_A(\mathcal L) = 4$.
Adding the red dotted loop would connect two existing crossing connected components, so $\com_A(\mathcal L\cup\{l_{\text{red}}\})=3<\com_A(\mathcal L)$. Adding the blue dotted loop would create a new crossing connected component, so $\com_A(\mathcal L\cup\{l_{\text{blue}}\})=5>\com_A(\mathcal L)+\com_A(\{l_{\text{blue}}\}).$
}
\label{pic:comp}
\end{figure}
In general, we no longer have monotonicity and subadditivity as in \eqref{eq:property_cross} for the component number: adding a new loop may connect two crossing components, resulting in $\com_A(\mathcal L_1\cup\mathcal L_2)<\com_A(\mathcal L_1)$; they may also create new components by collaboration, causing $\com_A(\mathcal L_1\cup\mathcal L_2)>\com_A(\mathcal L_1)+\com_A(\mathcal L_2)$, see Figure \ref{pic:comp}.

The \emph{cluster number} $\clu_A(\ls)$ is defined as the number of outermost clusters of $\ls$ which cross $A$, see e.g. Figure \ref{pic:clus}. 
It is immediate that for any loop ensemble $\ls$,
\[
\com_A(\ls) \ge \clu_A(\ls),
\]
and that the cluster number does not have monotonicity and subadditivity with respect to loop ensembles neither. 

Finally, if we fix an arbitrary loop ensemble $\mathcal L$, then all three quantities have monotonicity with respect to annuli, i.e. for $A'\subset A$, 
\[
(\cro_{A'}(\mathcal L),\com_{A'}(\mathcal L),\clu_{A'}(\mathcal L))\ge(\cro_{A}(\mathcal L),\com_{A}(\mathcal L),\clu_{A}(\mathcal L)).
\]

\begin{figure}[ht]
\centering
\includegraphics[width=0.5\textwidth]{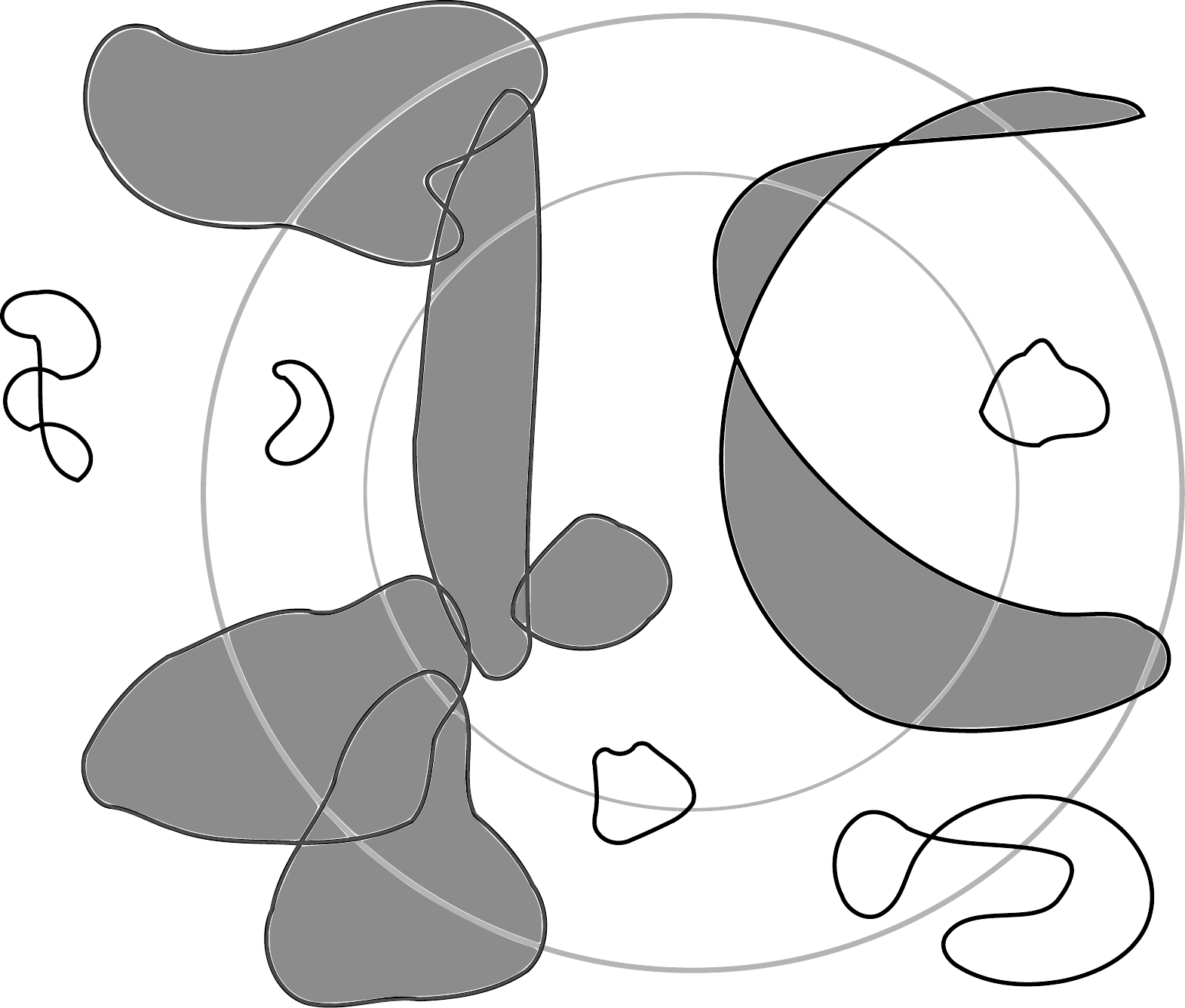}
\caption{
There are 2 crossing clusters in gray, and $\clu_A(\mathcal L) = 2$.}
\label{pic:clus}
\end{figure}

\subsection{The Brownian loop measure}\label{sec:Brownian loop measure}
Consider a simply connected domain $\Omega\subseteq \mathbb C$. 
The \emph{Brownian loop measure} in $\Omega$ was introduced by Lawler and Werner in \cite{lawler-werner}, and employed to construct $\cle$ in \cite{sheffield-werner}. Let $\mu^t_{x,\Omega}$ be the sub-probability measure on the set of paths in $\Omega$ started from $x\in\Omega$, defined from the probability distribution of a Brownian motion started at $x$ on the time interval $[0,t]$, which is killed upon hitting $\partial\Omega$. 
From this we obtain by disintegration the measures $\mu^{t}_{x\rightarrow y,\Omega}$ on paths from $x$ to $y$ inside $\Omega$,
\[
\mu^t_{x,\Omega}=\int_{\Omega}\mu^t_{x\rightarrow y,\Omega}d^2y,
\]
where $d^2y$ denotes the Lebesgue measure. Then the \emph{Brownian loop measure} on $\Omega$ is defined by the following integration: (here we choose the same normalization as in \cite{sheffield-werner}, which is one half of the Brownian loop measure defined in \cite{lawler-werner} considering the orientation)
\[
\mu^{\text{loop}}_{\Omega}=\int_0^{\infty}\frac{dt}{ 2 t}\int_{\Omega}\mu^t_{x\rightarrow x,\Omega}d^2y.
\]

Notice that it induces a measure on the traces of
unrooted loops by forgetting the root $x$ and the time-parametrization. Considering the fact that Brownian motion is invariant under conformal isomorphism up to a time change, the Brownian loop measure is also conformally invariant because of the time weight which appears in $\mu^{\text{loop}}_\Omega$. It is not hard to see from the definition that the Brownian loop measure satisfies the restriction property: if $\Omega'\subset\Omega$, then $\lm_{\Omega'}$ is the restriction of $\lm_\Omega$ to the set of loops in $\Omega'$.

Under the Brownian loop measure, the total mass of loops 
in the whole complex plane $\mathbb{C}$ is infinite (for all positive $R$, both the mass of loops of diameter greater than $R$ and the mass of loops of diameter smaller than $R$ are infinite), which can be viewed as a consequence of the conformal (scaling) invariance. However, for all $r<R$, the mass of the set of loops which stay in $\H$ intersecting both $r\mathbb{D}$ and $\mathbb{C}\setminus R\mathbb{D}$ is finite, where $\mathbb{D}$ is the unit disk, see the proof of Lemma 13 in \cite{lawler-werner}. This is also true for the Brownian loop measure on any subdomain of $\mathbb{H}$ by the restriction property (see eg. p.5 \cite{lawler-werner}).

\subsection{Loop-soup construction of CLE}\label{sec:loop-soup} Nested conformal loop ensemble 
$\cle_\kappa(\Omega)$ for $\kappa\in(8/3,4]$ defined on a simply connected domain $\Omega$ is a random collection of disjoint \emph{simple} loops in $\Omega$ characterized by the following properties:
\begin{itemize}
\item (Conformal invariance) If $\varphi:\Omega\rightarrow\Omega'$ is a conformal map from $\Omega$ onto $\Omega'$, then $\varphi(\cle_\kappa(\Omega))$ has the same distribution as $\cle_\kappa(\Omega')$.
\item (Restriction) If $U$ is a simply connected subset of $\Omega$ and $\Tilde{U}$ is obtained by removing from $\Omega$ all the $\cle_\kappa(\Omega)$ loops and their interior that do not entirely stay in $U$, then in each connected component $U'$ of the interior of $\Tilde{U}$, the conditional law of the set of loops that lie entirely in $U'$ is distributed as $\cle_\kappa(U')$.
\item (Locally finiteness) For each $\epsilon>0$, only finitely many loops have a diameter greater than $\epsilon$.
\item (Nesting) Conditioned on a loop $\gamma$ in $\cle_\kappa(\Omega)$ and all loops outside $\gamma$, the set of loops inside $\gamma$ is an independent $\cle_\kappa(\Omega_\gamma)$, where $\Omega_\gamma$ is the interior (finite) domain bounded by Jordan curve $\gamma$.
\end{itemize} 

A \emph{Brownian loop soup} $\bls^\lambda(\Omega)$ with intensity $\lambda$ is a Poissonian sample on the set of loops with intensity $\lambda\mu^{\text{loop}}_{\Omega}$ for $\lambda\in (0,1]$, which is characterized by the following properties:
\begin{itemize}
	\item The loop cluster is not unique and not boundary-touching, i.e. $\overline C\cap\partial\Omega=\emptyset$ almost surely.
	\item For any two disjoint measurable sets of loops $\ls_1$ and $\ls_2$, $\bls^\lambda(\Omega)\cap\ls_1$ and $\bls^\lambda(\Omega)\cap\ls_2$ are independent. In particular, if $\Omega'$ is a subdomain of $\Omega$, then $\bls^\lambda(\Omega)$ can be decomposed into two independent parts: $\bls^\lambda(\Omega')$ (the set of loops staying in $\Omega'$, which is again {a Brownian loop soup in $\Omega'$}) and $\bls^\lambda(\Omega')^\bot$ (the set of loops intersecting $\Omega\setminus\Omega'$).
	\item If $\varphi:\Omega\rightarrow\Omega'$ is a conformal isomorphism between two domains $\Omega$ and $\Omega'$, then $\varphi(\bls^\lambda(\Omega))=\{\varphi(l):l\in\bls^\lambda(\Omega)\}$ is distributed as $\bls^\lambda(\Omega')$.
    \item For any measurable set $\ls$ such that $\lambda\mu^{\text{loop}}_{\Omega}(\ls)<\infty$, the law of the number of elements in $\bls^\lambda(\Omega)\cap\ls$ satisfies the Poisson law with mean $\lambda\mu^{\text{loop}}_{\Omega}(\ls)$.
\end{itemize}

For a sample of Brownian loop soup $\bls^\lambda(\Omega)$ with intensity $\lambda$, as in Section \ref{sec:def_number}, denote by 
\[
\partial F(\bls^\lambda(\Omega))=\{\partial F(C): C \text{ is a cluster and there exists no cluster }C' \text{ such that }C\subseteq F(C')\}
\]
the set of boundaries of fillings (the complement of the unbounded
connected component of $\mathbb{C}\setminus C$) of all outermost clusters $C$ of $\bls^\lambda(\Omega)$. Then it is showed in \cite[Section 1.3]{sheffield-werner} that
$\partial F(\bls^\lambda(\Omega))$ has the same distribution as the non-nested $\cle_\kappa(\Omega)$ with $\lambda=(3\kappa-8)(6-\kappa)/2\kappa$. In particular, we have that for $\kappa\in(\frac{8}{3},4]$,
\begin{equation}
\label{eq:CLEtoBLS}
\cro_A(\cle_\kappa(\Omega))=2\com_A(\cle_\kappa(\Omega))\overset{d}{=}2\com_A(\bls^\lambda(\Omega))
\end{equation}
for any annulus $A$ and simply connected domain $\Omega$.

\subsection{Outline of the proof}

Here we present the intuition behind the proof of Theorem \ref{thm:main}. To begin with, by \eqref{eq:cro-comp} and \eqref{eq:CLEtoBLS}, it suffices to study $\com_{A(r,R)}(\bls^\lambda(\Omega))$. 
%Since Theorem \ref{thm:main} cannot be simply deduced from the domain Markov property of CLEs, as explained in Remark \ref{rmk:thm_main}, 
Then we divide $\ls=\bls^\lambda(\Omega)$ into $\ls_1=\bls^\lambda_{< a}(\Omega)$ (loops with diameter less than $a$) and $\ls_2=\bls^\lambda_{\ge a}(\Omega)$ (loops with diameter larger or equal to $a$), and reduce the problem to $\com_{A(r,R)}(\ls_1)$ and $\com_{A(r,R)}(\ls_2)$ by Lemma \ref{lm:decomposition}.
% according to their diameters, in order to upper-bound $\com(\ls_1)$ by $\clu(\ls_1)$, see Lemma \ref{lm:CompToCluster}.

Intuitively, loops with small diameter cannot appear in many different crossing connected components of $A(r,R)$, which inspired us to bound $\com(\ls_1)$ by $\clu(\ls_1)$ in Lemma \ref{lm:CompToCluster}.
%of crossing clusters and $A(r,R)$. 
The main technicality in this paper consists of dealing with $\clu(\ls_1)$, which will be discussed in Section \ref{sec:3.3} and Section \ref{sec:iter} by establishing a recursive inequality using the conformal invariance of the Brownian loop soup.

Moreover, the probability distribution on the number of loops in $\ls_2$ has super-exponential tail since it is a Poisson distribution. Combined with Fomin's identity for non-intersection probabilities for the Brownian paths, we obtain the probabilistic super-exponential decay of $\cro(\ls_2)$ in Proposition \ref{prop:loop_crossing}.

In conclusion, the desired upper bound for $\com(\bls^\lambda(\Omega))=\com(\ls_1\cup\ls_2)$ follows from  estimates of $\clu(\ls_1)$ and $\cro(\ls_2)$. 
%These steps shall be validated in Section \ref{sec:comp} and Section \ref{sec:clus}. 
We remark that the annuli subscripts in the above notions of crossing/component/cluster numbers are deliberately omitted, because we need to change the annuli slightly in each step.

Finally, in Section \ref{sec:general_ann}, we prove Corollary \ref{cor:main2} based on the estimates established in Theorem \ref{thm:main}. In the last Section \ref{sec:complexity}, we carefully apply Corollary \ref{cor:main2} to the setup of the complexity for the convergence of probabilities of cylindrical events for double-dimer configurations.

\section{Component Number and Cluster Number}\label{sec:comp_clus_number}
The goal of this section is to explore some deterministic properties and relations of $\com_A(\ls)$ and $\clu_A(\ls)$ and present the proof of the super-exponential decay of ${\sup_{U\subseteq\H}}\mathbb P\left[\clu_{A(r,R)}(\bls^\lambda_{<a}({U}))\ge n\right]$, $\lambda\le 1, a>0$, see Proposition \ref{prop:clusterMajor}. Here we assume the annuli to be centered at $0$ without loss of generality but keep in mind that those relations are translationally invariant. Besides, we only consider loop ensembles with the following properties: for any fixed $r>0$,
\begin{equation}\label{eq:ls}
\begin{aligned}
&\text{All loops in $\mathcal L$ do not touch (i.e., do not intersect without crossing) $\partial\Omega$, $\partial B_r$ or any other loop in $\mathcal L$;}\\
&\text{Outermost boundaries of clusters of $\mathcal L$ do not touch $\partial\Omega$, $\partial B_r$ or any other loop in $\mathcal L$.}
\end{aligned}
\end{equation}

It is known that $\mathcal L=\bls^{\lambda}$ satisfies \eqref{eq:ls} almost surely \cite{sheffield-werner}. The assumptions \eqref{eq:ls} also holds for $\bls^{\lambda}_{<a},
\bls^{\lambda}_{\ge a}$, since there is a positive probability that $\bls^{\lambda}=\bls^{\lambda}_{<a}$, and $\bls^{\lambda}_{<a}$ is independent of $\bls^{\lambda}_{\ge a}$. 

\subsection{Component number}\label{sec:comp}

Recall that the component number $\com_A(\ls)$ is the number of connected components of $\cup_{C\in\{\text{outermost clusters of }\ls\}} F(C)\cap A$ connecting $\partial B(r)$ and $\partial B(R)$. We first show that for any crossing connected component of $F(C)\cap A$, there is a finite collection of loop arcs whose union crosses $A$ inside $D$.
\begin{lemma}\label{lm:LoopPath}
Let $\ls$ be a loop ensemble satisfying \eqref{eq:ls}. For each annulus $A(r,R)$ and crossing connected component $D$, there exists a path $\gamma\subset D$ comprised of finitely many arcs of loops in $\ls$, such that $\gamma$ crosses $A(r,R)$. This sequence of loops will be denoted by $L_\gamma$.
\end{lemma}
\begin{proof}
Denote by $C$ the cluster such that $D\subset F(C)$. By \eqref{eq:ls}, clusters and loops cannot touch $\partial A(r,R)$. Thus there exist loops $l,l'\in C$ such that $l,l'$ intersect $\partial B(r)\cap D$ and $\partial B(R)\cap D$ respectively.

Since $l,l'$ are in the same cluster $C$,  there exists a finite chain of loops $l_0=l,l_1,l_2,\ldots,l_n=l'$ in $\ls$ such that $l_i$ and $l_{i+1}$ are adjacent. We conclude that $\cup_{i=1}^nl_i \cap D$ {crosses $A(r,R)$} since $F(C)$ is simply connected (otherwise the union of $D$ with all fillings of chains of loops connecting $l$ and $l'$ encircles a hole). Thus we can draw a crossing path $\gamma$ out of a crossing chain of finite loops.
\end{proof}

Using Lemma \ref{lm:LoopPath} for the decomposition of the loop ensemble, the component number can be bounded above by the component number of a smaller annulus as follows. 

\begin{lemma}\label{lm:decomposition}
Let $\ls_1,\ls_2$ be two disjoint loop ensembles satisfying \eqref{eq:ls}. Take $0<r< r'< R'<R$, then
\begin{equation}\label{eq:decomp}
\begin{aligned}
\com_{A(r,R)}(\ls_1\cup\ls_2) \le &\com_{A(r',R')}(\ls_{1})+\cro_{A(r,r')}(\ls_2)+\cro_{A(R',R)}(\ls_2)\\
    &+\#\{l\in\mathcal L_2:l\cap A(r',R')\ne\emptyset ,\,l\subset A(r,R)\}.
\end{aligned}
\end{equation}
In particular, if $\ls$ is inside $\H$, $\ls_1=\ls(\Omega)$ and $\ls_2=\ls(\Omega)^\bot$ for some domain $A(r,R)\cap \H \subset \Omega$, then
\[
\com_{A(r,R)}(\ls)\le \com_{A(r',R')}(\ls(\Omega))+\cro_{A(r,r')}(\ls(\Omega)^\bot)+\cro_{A(R',R)}(\ls(\Omega)^\bot).
\]
\end{lemma}

\begin{proof}
For each component $D$ that contributes to $\com_{A(r,R)}(\ls_1\cup\ls_2)$ which also crosses $A(r',R')$, it follows from Lemma \ref{lm:LoopPath} that there is a path $\gamma$ crossing $A(r',R')$ within $D\cap A(r',R')$ constituted by finitely many arcs of loops in $L_\gamma$, contained in $D$. 

If $L_\gamma$ is a subset of $\ls_{1}$, then it stays in a connected component which contributes to $\com_{A(r',R')}(\ls_{1})$. Otherwise, there exists $l\in{\ls_{2}}$ such that $l\cap \gamma\neq\emptyset$. In such cases, if $l\subset A(r,R)$, then it contributes to the term $\#\{l\in\mathcal L_2:l\cap A(r',R')\ne\emptyset ,\,l\subset A(r,R)\}$. If $l\not\subset A(r,R)$, then $l$ intersects $\partial B_r$ or $\partial B_R$, which contributes to $\cro_{A(R,R')}(\ls_2)$ or $\cro_{A(r,r')}(\ls_2)$ since $\gamma\subset A(r',R')$ and $l\cap \gamma\neq\emptyset$. The desired upper bound \eqref{eq:decomp} is thus proved since for distinctive crossing components $D_1,\ldots,D_n$ contributing to the left-hand side of \eqref{eq:decomp}, one can find different crossing components or loops contributing to the right-hand side of \eqref{eq:decomp}
contained in $D_1,\ldots,D_n$, respectively.
\end{proof}

\subsection{Cluster number}\label{sec:clus}
For any loop ensemble whose loops have diameter less than $a$, the component number in $A(r,R)$ can be bounded by the cluster number with respect to an annulus which is $a$-smaller than $A(r,R)$.

\begin{lemma}\label{lm:CompToCluster}
For $0<r<r+a<R-a<R$, let $\ls$ be a loop ensemble such that $\ls_{<a}$ satisfies \eqref{eq:ls}.
we have 
\[
\com_{A(r,R)}(\ls_{<a})\le \clu_{A(r+a,R-a)}(\ls_{<a}(A(r,R))).
\]
\end{lemma}
\begin{proof}
By Lemma \ref{lm:LoopPath}, for each component $D$ that contributes to $\com_{A(r,R)}(\ls_{<a})$, we can find a path $\gamma$ in $D\cap A(r+a,R-a)$ from loops in $L_\gamma \subset \ls_{<a}$. Since all loops in $\mathcal L_{<a}$ have diameter less than $a$, $L_\gamma$ is contained in $A(r,R)$. Therefore, $L_\gamma$ is a subset of a cluster which contributes to $\clu_{A(r+a,R-a)}(\ls_{<a}(A(r,R)))$. Conversely, this cluster is connected and stays within $A(r,R)$, thus it is contained in $D$, which gives the injectivity of the mapping from $\com_{A(r,R)}(\ls_{<a})$ to $\clu_{A(r+a,R-a)}(\ls_{<a}(A(r,R)))$.
\end{proof}

Similarly as Lemma \ref{lm:decomposition}, we obtain the following upper bound for the cluster number.
\begin{lemma}\label{lm:cluster_decomp}
Let $0<r\le r'<R'\le R$, and $\ls_1,\ls_2$ be two disjoint loop ensembles satisfying \eqref{eq:ls}, then
%If $\ls=\ls_1\sqcup\ls_2$,
\begin{equation*}
\begin{aligned}
\clu_{A(r,R)}(\ls_1\cup\ls_2) \le &\clu_{A(r',R')}(\ls_{1})+\#\{l\in\ls_2:l\cap A(r',R')\ne\emptyset ,\,l\subset A(r,R)\}\\
&+\#\{l\in\ls_2:l\text{ crosses }A(r,r')\text{ or }A(R',R)\}.
\end{aligned}
\end{equation*}
In particular, 
\[
	\clu_{A(r,R)}(\ls_1\cup\ls_2)\le \clu_{A(r,R)}(\ls_{1})+\#\{l\in\ls_2:l\cap A(r,R)\neq\emptyset\}
\]
in the degenerate case $r'=r, R'=R$.
\end{lemma}
\begin{proof} As in the proof of Lemma \ref{lm:decomposition}, if in the beginning we take any cluster $C$ in $\clu_{A(r,R)}(\ls_1\cup\ls_2)$, we can decompose the cluster number depending on whether $\ls_1$ restricted to $C$ gives a crossing of $A(r',R')$ or not. Then the argument follows the same line as the proof of Lemma \ref{lm:decomposition}.
\end{proof}

Let us briefly mention how results in this section will be used in the probabilistic setting for Poissonnian Brownian loops to prove the quasi-multiplicativity of crossing probabilities. Recall that $\bls^\lambda(A(r,R))$ is a Brownian loop soup with intensity $\lambda\in(0,1]$ in $A(r,R)$, and for simplicity, we will write $\bls^\lambda(r,R)=\bls(A(r,R))$. Let $\rho<r<r'<\rho'<R'<R<P$ and $\epsilon,s>0$. In the next paragraph, we give an upper bound on $\clu_{A(r,R)}(\bls^\lambda(\rho,P))$. 

Firstly, we can upper-bound this cluster number of $A(r,R)$ by the cluster number of $A(r,r')$ and $A(R',R)$, which follows from Lemma \ref{lm:cluster_decomp} that 
\begin{align*}
\clu_{A(r,R)}(\bls^\lambda(\rho,P))\le & \min\{\clu_{A(r,r')}(\bls^\lambda(\rho,\rho')),\clu_{A(R',R)}(\bls^\lambda(\rho',P))\}\\
&+\#\{l\in\bls^\lambda(\rho,P):l\text{ crosses }A(r',\rho')\text{ or }A(\rho',R')\}.
\end{align*}
By the independence of $\bls^\lambda(\rho,\rho')$, $\bls^\lambda(\rho',P)$ and the Poisson tail of \[
\#\{l\in\bls^\lambda(\rho,P):l\text{ crosses }A(r',\rho')\text{ or }A(\rho',R')\},
\]
we have that
\begin{align}\label{eq:multicativity}
	\mathbb P\left[\clu_{A(r,R)}\left(\bls^\lambda(\rho,P))\right)\ge n\right] \le &\mathbb P\left[\clu_{A(r,r')}\left(\bls^\lambda(\rho,\rho')\right)\ge(1-\epsilon)n\right] \\
&\times\mathbb P\left[\clu_{A(R',R)}\left(\bls^\lambda(\rho',P)\right)\ge(1-\epsilon)n\right]+O(s^n). \nonumber
\end{align}
The inequality \eqref{eq:multicativity} is a key component for proving the recursive relation \eqref{eq:iter}, which will result in the desired super-exponential decay.

\subsection{Super-exponential decay of the cluster number}
\label{sec:3.3}
%Recall that given loop ensemble $\ls$, the cluster number with respect to an annulus is the number of crossing clusters of $\ls$. 

In this subsection, we prove that the probability distribution on the cluster number has a super-exponentially tail. It is intuitively not hard to see that crossing clusters occur "disjointly" in a loop ensemble, %which is a Poissonnian sample of loops. 
therefore the probability of finding two crossing clusters should be smaller than the product of their probabilities.

%In this part, we will restrict ourselves to {\red subdomains of $\H$.} 
%The general quad case will be discussed in Section \ref{sec:general_ann}. 
%{\red The following estimate gives an intuition and is crucial for the super-exponential decay of the crossing estimates for simple CLEs.}
\begin{proposition}\label{prop:clusterMajor}
Let $0<a<r<1<R$. Denote by $\bls^\lambda_{<a}({U})$ the set of loops with diameter less than $a$ in a Brownian loop soup with intensity $\lambda\in(0,1]$ {in any open set $U\subseteq\H$}. Then for each $s>0$, we have 
\begin{equation}\label{eq:clusterMajor}
{\sup_{U\subseteq\H}}\mathbb P\left[\clu_{A(r,R)}(\bls^\lambda_{<a}({U}))\ge n\right] = O(s^n)\, \text{ as } n\rightarrow\infty,
\end{equation}
where the supremum is taken over all open subsets of $\H$, and the constant in $O(s^n)$ depends on $a,R/r,\lambda$ and $s$.
\end{proposition}

\begin{remark}
Different from Theorem \ref{thm:main} and Corollary \ref{cor:main2}, the supremum taken in Proposition \ref{prop:clusterMajor} is not restricted to simply connected domains. This is validated by the flexibility of the construction of the Brownian loop soup, and it helps to simplify
the discussion on the distortion in conformal mappings used in the proof of Proposition \ref{prop:clusterMajor}.
\end{remark}

\begin{proof}[Strategy of the proof of Proposition \ref{prop:clusterMajor}]
Let us define
\begin{equation}\label{eqdef:f(n)}
f(n):=\sup_{U\subseteq\mathbb H}\mathbb P\left[\clu_{A(r,R)}(\bls^\lambda_{<a}(U))\ge n\right].
\end{equation}
We estimate $f(n)$ inductively, where the step of induction can be described as follows. Note that intuitively, conditioned on having $n$ crossing clusters, one can expect two scenario. In the first one, the space remaining to accommodate one more crossing cluster becomes less and less, leading to a multiplying factor tending to $0$. In the second scenario, all $n$ crossing clusters cross $A(r,R)$ inside a strictly smaller subset 
\[
A^{(\eta)}(r,R):=\{z\in A(r,R): 0<\arg z<\eta<\pi\}
\]
for some fixed $\eta$, depending only on $s$. 
Then, we can conformally map $A^{(\eta)}(r,R)$ to the annulus $A(r',R')$ with $r'<r<1<R<R'$ and, by conformal invariance, get a sample of the Brownian loop soup having $n$ clusters crossing $A(r',R')$. A technical analysis shows that the probability to have such a sample can be upper-bounded by $cq^n\cdot f((1-\epsilon)n)+O(s^{2n})$. As a result we find out that for all $s,\epsilon \in(0,1)$, we can find some $c>0$, $q<1$ and any $\epsilon>0$, the following holds:
\begin{equation}\label{eq:iter}
	f(n+1) \le \frac{s}{2}f(n)+cq^n\cdot f((1-\epsilon)n)+O(s^{2n}).
\end{equation}
Let us mention again here the constants in $O(s^{2n})$ depend on $\epsilon$ and $s$. We claim that \eqref{eq:iter} is sufficient for deducing Proposition \ref{prop:clusterMajor}. 
In fact, if \eqref{eq:iter} holds, we can take $\epsilon$ small enough such that $s^{2\epsilon}>q$. Note that for $n$ large enough, $\frac{cq^n}{s^{\epsilon n+1}}<\frac{1}{2}$. Then \eqref{eq:iter} divided by $s^{n+1}$ gives that
\[
\frac{f(n+1)}{s^{n+1}}\le \frac{1}{2}\frac{f((1-\epsilon)n)}{s^{(1-\epsilon)n}}+\frac{1}{2}\frac{f(n)}{s^n}+O(s^{n-1}),
\]
which implies that $\frac{f(n)}{s^{n}}$ is bounded for all $n$, hence the super-exponential decay of $f(n)$. 

Together with \eqref{eq:iter} and \eqref{eqdef:f(n)}, this completes the proof of Proposition \ref{prop:clusterMajor} modulo the technical proof of \eqref{eq:iter}, which is postponed to Section \ref{sec:iter}. 
\end{proof}

The following result on the probability of the existence of a crossing cluster inside a (conformally) thin tube will be used in Section \ref{sec:iter}.

\begin{figure}[ht]
\centering
\includegraphics[width=0.5\textwidth]{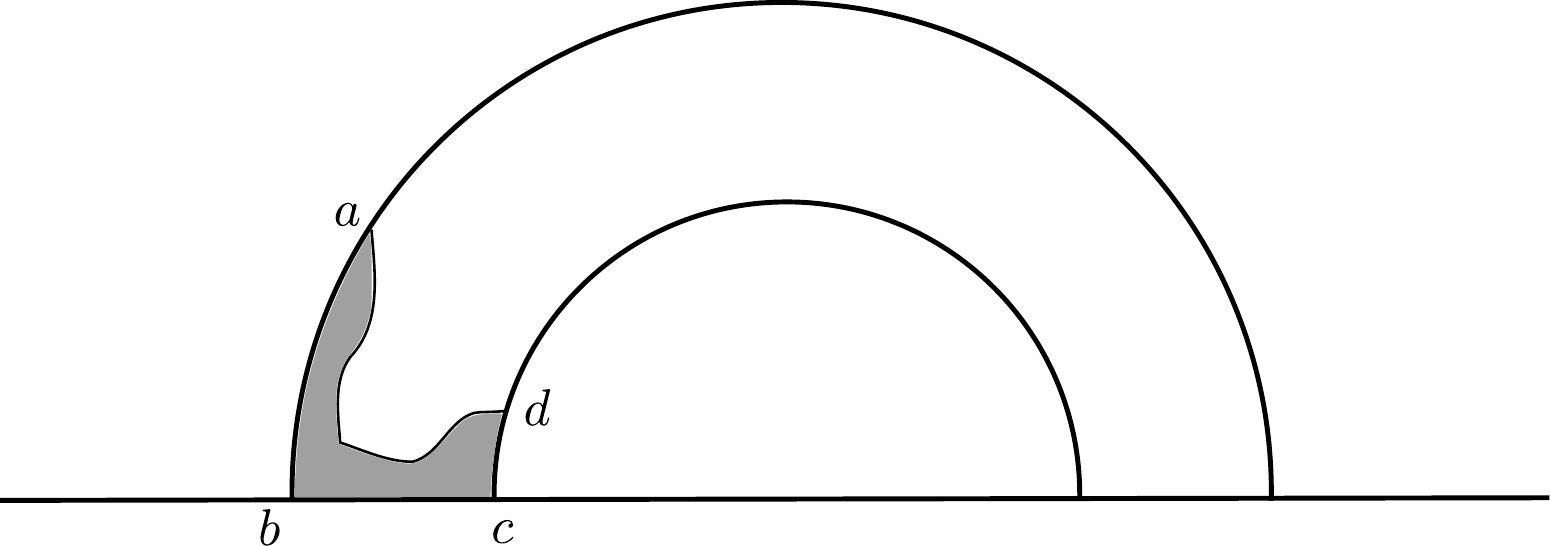}
\caption{An illustration for $(Q;a,b,c,d)$ in Lemma \ref{lm:narrow_crossing}.}
\end{figure}

\begin{lemma}\label{lm:narrow_crossing}
	For any $\epsilon>0$ and $0<r<R$, there exists $\delta>0$ such that uniformly for all crossing-quads inside $A(r,R)$ of the form $(Q;a,b,c,d)$ with $b=-R$ and $c=-r$, such that 
	\[(ab)\subset \partial B_R, (bc)\subset\mathbb R_-, (cd)\subset \partial B_r\text{ and }\inf_{z\in(bc),w\in(ad)}|z-w|<\delta,
	\] we have
	\begin{equation}\label{eq:narrow_crossing}
	\mathbb P[(ab) \text{ and }(cd) \text{ are }\text{connected by a chain of loops in }\bls^\lambda(\H) \text{ not touching }(bc)\text{ and }(ad)]<\epsilon.
	\end{equation}
\end{lemma}
\begin{proof}
	Suppose the contrary, then there exists a sequence of quads $(Q_\delta;a_\delta, b_\delta,c_\delta, d_\delta)\subset A$ satisfying the same conditions as in the statement, such that the probability that $(a_\delta b_\delta )$ and $(c_\delta d_\delta )$ are connected by a chain of loops in $\bls^\lambda(\H)$ not touching $(b_\delta c_\delta)$ and $(a_\delta d_\delta)$ is uniformly away from $0$. By Kochen-Stone lemma, with positive probability, we can find a sequence of  clusters of $\bls^\lambda(\H)$ arbitrarily close to $\mathbb R_-$. These clusters are of diameter larger than $R-r$, which is not possible in the sub-critical regime of the Brownian loop soup with intensity $\lambda \mu^{\text{loop}}_\Omega$, $\lambda\in(0,1]$, see e.g. \cite[Lemma 9.7]{sheffield-werner}. Thus by contradiction we have \eqref{eq:narrow_crossing}.
\end{proof}

\subsection{Proof of the recursive inequality \texorpdfstring{\eqref{eq:iter}}{}.}\label{sec:iter}
Throughout this section, we fix the intensity of the Brownian loop soup in \eqref{eq:iter} to be some $\lambda\in(0,1]$ and omit it. Before diving into the technical details of the proof, let us first explain the choice of parameters.
%, which is a rather delicate matter
For all $A(r,R)$, denote the sector of angle $\eta$ by
\[
A^{(\eta)}(r,R) :=A(r,R)\cap\{z\in\H: 0<\arg z<\eta\}.
\] 
For any open subset $U\subset\H$, denote the Brownian loop soup on top of it by
\[
\bls^{(\eta)}(U)=\bls(A^{(\eta)}(0,\infty)\cap U),
\]
with the mnemonics 
\[
A(r,R)=A^{(\pi)}(r,R),\,
\bls(U)=\bls^{(\pi)}(U).\]
For each \emph{fixed} $s$, we first choose $\eta$ sufficiently close to $\pi$ such that the probability of having a cluster in $\mathcal B(\H)$ which crosses $A(r,R)$ inside a quad $
(Q;a,b,c,d)$ with the arc $(ab)\subset \partial B_R$, $(bc)\subset \partial B_r$, $(cd)\subset \partial B_r$ and $(ad)$ not contained in $A^{(\eta)}(r,R)$ is less than $\frac{s}{2}$ by Lemma \ref{lm:narrow_crossing}. For all $n\in\N$, conditioned on the event that $n$ crossing clusters cross $A(r,R)$ inside $A^{(\eta)}(r,R)$, a family of radii is required for applying Lemma \ref{lm:cluster_decomp}.
%to relate the cluster number of the $A^{(\eta)}(r,R)$ to the cluster number of a sub-sector with the same conformal modulus as $A(r,R)$, which by conformal invariance of the Brownian loop measure, will give the same probability of having $n$ crossing clusters of $A(r,R)$ modulo an error term that decays super-exponentially. 

Due to the scaling invariance of the Brownian loop soup, we suppose without loss of generality that $0<r<1<R$. Define
\begin{equation}\label{rR}
\begin{array}{lcr}
    r_\beta = r^{\frac{(1-\beta)\pi+\beta\eta}{\eta}}, &R_\beta = R^{\frac{(1-\beta)\pi+\beta\eta}{\eta}} &\text{ if } \beta\in[0,1]\\
	r_\beta = r^{\frac{(2-\beta)
	\pi+(\beta-1)\eta}{\pi}}, &R_\beta=R^{\frac{(2-\beta)\pi+(\beta-1)\eta}{\pi}} &\text{   if } \beta\in[1,2].
\end{array}
\end{equation}

Note that $r_1=r$, $R_1=R$, $r_\beta$ is increasing in $\beta$ and $R_\beta$ is decreasing in $\beta$. Therefore, $A(r_{\beta_1},R_{\beta_1})\subset A(r_{\beta_2},R_{\beta_2})$ if $\beta_1>\beta_2$. See {Figure \ref{pic:radii}} for an illustration.

For each open subset $U\subseteq \H$, conditioned on the event that 
$\clu_{A(r,R)}(\bls_{<a}(U))\ge n$, we can order the clusters counterclockwise by their rightmost crossing connected components, and denote by $D_1,\ldots,D_n$ the first $n$ components, from right to left in $A(r,R)$, see e.g. Figure \ref{pic:D_n}. Denote by $E_{n,\eta}(U)$, $\tilde{E}_{n,\eta}(U)$ the events that 
\begin{equation}
\begin{aligned}
E_{n,\eta}(U)&:= \{\bls_{<a}(U)\text{ has }n\text{ crossing clusters and }D_n \text{ is inside }A^{(\eta)}(r,R)\}. \\
\tilde{E}_{n,\eta}(U)&:= \{\bls_{<a}(U)\text{ has }n\text{ crossing clusters and }D_n \text{ is not contained in }A^{(\eta)}(r,R)\}.
\end{aligned}
\end{equation}
Note that 
\[
\tilde{E}_{n,\eta}(U)\cup E_{n,\eta}(U)= \{\clu_{A(r,R)}(\bls_{<a}(U))\ge n\} \text{, }\sup_{U\subseteq \H}\mathbb{P}[\tilde{E}_{n,\eta}(U)]\le f(n)
\]
and conditioned on $E_{n,\eta}(U)$, it may happen that the $n$-th cluster is not contained in $A^{(\eta)}(r,R)$. Now we can embark on the proof of the recursive inequality \eqref{eq:iter}.

\begin{figure}[ht]
\centering
\includegraphics[width=.5\textwidth]{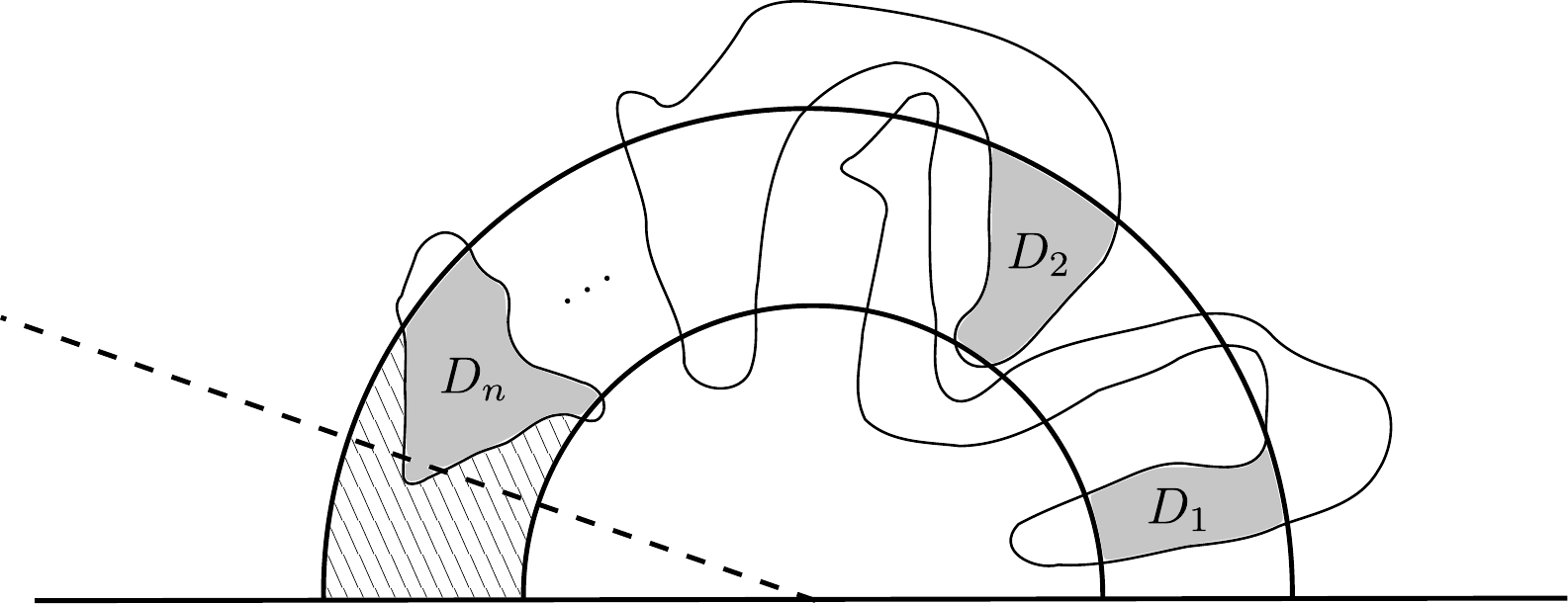}
\caption{
An illustration of the rightmost components $(D_i)$ of crossing clusters. On the event $\tilde{E}_{n,\eta}$, $D_n$ must intersect $A(r,R)\backslash A^{(\eta)}(r,R)$. If $D_{n+1}$ exists, then it must live in the shaded area.}
\label{pic:D_n}
\end{figure}

\vspace{2mm}
\noindent\emph{Step 1: Decompose the crossing probability.} 
\vspace{1mm}
Let us decompose $f(n+1)$ with respect to $E_{n,\eta}(U)$ and $\tilde{E}_{n,\eta}(U)$. Assume that $\clu_{A(r,R)}(\bls_{<a}(U))\ge n+1$ and $\tilde{E}_{n,\eta}(U)$ happens. By definition, this means that the rightmost crossing connected component $D_n$ of the $n$-th cluster is not within $A^{(\eta)}(r,R)$, which implies that the $(n+1)$-th cluster crosses $A(r,R)$ inside some crossing quad that satisfies the assumptions of Lemma \ref{lm:narrow_crossing} with $\epsilon=\frac{s}{2}$ (by the choice of $\eta$),  as illustrated in Figure \ref{pic:D_n}. Conditioned on the event $\tilde{E}_{n,\eta}$, the loops outside the clusters to which $D_1,\ldots,D_n$ belong is an independent Brownian loop soup. Then if in addition $D_n$ intersects $A(r,R)\setminus A^{(\eta)}(r,R)$, it follows from Lemma \ref{lm:narrow_crossing}
\[
\sup_{U\subseteq\H}\mathbb{P}\left[\clu_{A(r,R)}(\bls_{<a}(U))\ge n+1|\tilde{E}_{n,\eta}\right]\le \frac{s}{2}.
\]
Therefore,
\begin{equation}\label{eq:prob_decomp}
\begin{aligned}
f(n+1)=&\sup_{U\subseteq\mathbb H}\mathbb P\left[\clu_{A(r,R)}(\bls_{<a}(U))\ge n+1\right] \\ \le
&\sup_{U\subseteq\mathbb H}\left(\mathbb P\left[\tilde{E}_{n,\eta},\clu_{A(r,R)}(\bls_{<a}(U))\ge n+1\right] + \mathbb P\left[E_{n,\eta}(U)\right]\right) \\ \le
&\sup_{U\subseteq\mathbb H} \mathbb P\left[\clu_{A(r,R)}(\bls_{<a}(U))\ge n+1|\tilde{E}_{n,\eta}(U)\right]\mathbb{P}[\tilde{E}_{n,\eta}(U)] + \sup_{U\subseteq\mathbb H}\mathbb P\left[E_{n,\eta}(U)\right]  \\ \le
&\frac{s}{2}\cdot f(n) + \sup_{U\subseteq\mathbb H}\mathbb P\left[E_{n,\eta}(U)\right],
\end{aligned}
\end{equation}

\noindent\emph{Step 2: Decompose the cluster number in} $\mathbb P\left[E_{n,\eta}(U)\right].$ %In this step, we fix an open subset $U\subseteq\H$ and simply write $E_{n,\eta}(U)$ instead of $E_{n,\eta}$. 
In this step we aim to show the following alternative: if $E_{n,\eta}$ happens, either the restricted Brownian loop soup $B^{(\eta)}_{<a}(U)$ has at least $(1-\epsilon)n$ clusters crossing a slightly thinner annulus $A^{(\eta)}(r_{1.5},R_{1.5})$, or we are in the setup to apply a Poisson tail estimate.
%We will show that the knowledge of rightmost components of clusters is sufficient for switching to a smaller domain $A^{(\eta)}(r^*,R^*)$ if the annulus to be crossed shrinks correspondingly.

Similarly to the proof of Lemma \ref{lm:cluster_decomp}, for any crossing cluster $C$ from $\bls_{<a}(U)$ whose rightmost crossing component $D$ stays in $A^{(\eta)}(r,R)$, it follows from Lemma \ref{lm:LoopPath} that $D$ contains a path $\gamma$ crossing $A^{(\eta)}(r_{1.5};R_{1.5})$ comprised of finitely many arcs of loops in $C$. If the loops in $L_\gamma$ (which give the arcs that constitute $\gamma$) are part of $\bls^{(\eta)}_{<a}(U)$, then $C$ contains a crossing cluster of $A(r_{1.5},R_{1.5})$. Otherwise, we can find a loop $l_C$ in $C$ that intersects both $A^{(\eta)}(r_{1.5},R_{1.5})$ and $U\backslash U^{(\eta)}$, where $U^{(\eta)}=\{z\in U: 0<\arg z<\eta\}$. Recall that $D$ is contained $A^{(\eta)}(r,R)$, therefore $l_C$ crosses $A^{(\eta)}(r, r_{1.5})$ or $A^{(\eta)}(R_{1.5},R)$ to reach $U\backslash U^{(\eta)}$.

Under $E_{n,\eta}(U)$, all components $D_1,\ldots,D_n$ lie in $A^{(\eta)}(r,R)$. Applying this argument to each cluster that $D_i, i=1,\ldots,n$ belongs to, we get that for all $\epsilon'\in(0,1)$,
\begin{equation}\label{eq:clus_decomp}
\begin{aligned}
&\mathbb P\left[E_{n,\eta}(U)\right]\\
\le& \mathbb P\left[\#\left\{l\in\bls_{<a}(U): l\text{ crosses }A^{(\eta)}(r, r_{1.5})\text{ or }A^{(\eta)}(R_{1.5},R) \right\} +\clu_{A^{(\eta)}(r_{1.5},R_{1.5})}\left(\bls^{(\eta)}_{<a}(U)\right)\ge n \right]\\ 
\le&\mathbb P\left[ \#\left\{ l\in\bls_{<a}(U): l\text{ crosses }A^{(\eta)}(r, r_{1.5})\text{ or }A^{(\eta)}(R_{1.5},R)\right \}\ge\epsilon' n\right]\\
&\qquad + \mathbb{P}\left[\clu_{A^{(\eta)}(r_{1.5},R_{1.5})} \left(\bls^{(\eta)}_{<a}(U)\right) \ge (1-\epsilon')n \right]\\
\le& \mathbb{P}\left[\clu_{A^{(\eta)}(r_{1.5},R_{1.5})} \left(\bls^{(\eta)}_{<a}(U)\right) \ge (1-\epsilon')n\right]+O(s^{2n}),
\end{aligned}
\end{equation}
where the last line follows from the fact that the term 
\begin{align*}
    &\# \left\{l\in\bls_{<a}(U): l\text{ crosses }A^{(\eta)}(r,r_{1.5})\text{ or }A^{(\eta)}(R_{1.5},R)\right\}\\
    \le &\# \left\{l\in\bls(\H): l\text{ crosses }A^{(\eta)}(r,r_{1.5})\text{ or }A^{(\eta)}(R_{1.5},R) \right\}
\end{align*}
has a super-exponentially decaying Poisson tail independent of $U$.

\begin{figure}[ht]
\centering
\includegraphics[width=.8\textwidth]{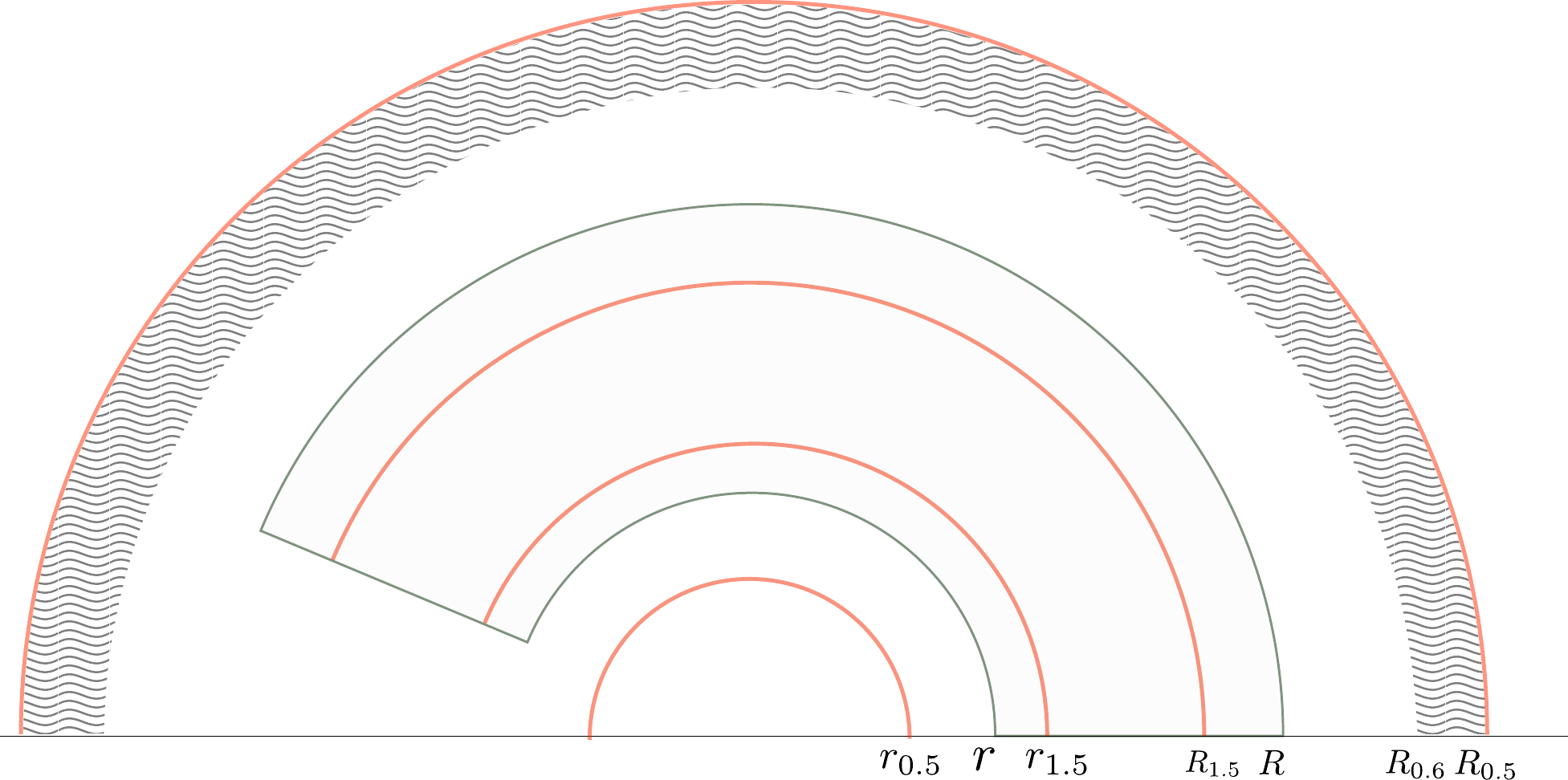}
\caption{The relation of radii defined in \eqref{rR} and corresponding annuli.}
\label{pic:radii}
\end{figure}

The recursive relation \eqref{eq:iter} then reduces to
\begin{equation}\label{eq:clus_eta}
	\sup_{U\subseteq \H}\mathbb P\left[\clu_{A^{(\eta)}(r_{1.5},R_{1.5})}\left(\bls^{(\eta)}_{<a}(U)\right)\ge (1-\epsilon')n \right]\le  cq^n\cdot\sup_{U\subseteq\H}\mathbb P[\clu_{A(r,R)}(\bls_{<a}(U))\ge(1- \epsilon)n]+O(s^{2n}).
\end{equation}

We will prove \eqref{eq:clus_eta} in the next two steps. In fact, it follows from \eqref{rR} that the conformal modulus of the quad $A^{(\eta)}(r_{1.5},R_{1.5})$ is strictly bigger than the conformal modulus of $A(r,R)\cap\H$, which is the main reason for the factor $q^n$ to appear on the right-hand side, see \eqref{eq:exp_decay}. This argument requires a careful justification because $\bls_{<a}(U)$ is not conformally invariant, which requires the constant $c$ (see \eqref{eq:c}) and the correction term $O(s^{2n})$ on the right-hand side of \eqref{eq:clus_eta}.

\vspace{2mm}
\noindent\emph{Step 3: Transform $A^{(\eta)}(r_{1.5},R_{1.5})$ to $A(r_{0.5},R_{0.5})$}.
Define the conformal map from $\H^{(\eta)}=\{z\in \H: 0<\arg z<\eta\}$ to $\H$  
\[
\phi_\eta:z=re^{i\theta}\mapsto r^{\frac{\pi}{\eta}}e^{i\frac{\theta\pi}{\eta}} \text{  for  }r>0, \theta\in(0,\eta),
\] 
then
\[
\phi_\eta (A^{(\eta)}(r_{1.5},R_{1.5})) = A(r_{0.5},R_{0.5}),
\]
hence,
\[
\clu_{A^{(\eta)}(r_{1.5},R_{1.5})}\left(\bls_{<a}^{(\eta)}(U)\right) =  \clu_{A(r_{0.5},R_{0.5})}\left(\phi_\eta(\bls^{(\eta)}_{<a}(U))\right).
\]
Only loops in $U\cap B_{R+a}$ contribute to the left-hand side of the above relation, therefore we assume without loss of generality that $U\subseteq B_{R+a}$. Then the conformal invariance of the Brownian loop measure and a simple computation on the distortion of $\phi_\eta$ give that there exist constants $0<c_1<1<c_2$ depending on $a,\eta,r,R$ such that almost surely
\begin{equation}\label{eq:conf_map}
\bls_{<c_1a}(\phi_\eta(U^{(\eta)})) \subseteq \phi_\eta\left(\bls^{(\eta)}_{<a}(U)\right) \subseteq \bls_{<c_2a}(\phi_\eta(U^{(\eta)})),
\end{equation}
where $U^{(\eta)}=\{z\in U: 0<\arg z<\eta\}$. Let $\ls':=\bls_{[c_1a,c_2a[}(\phi_\eta(U^{(\eta)}))$, a sample of Brownian loops within $\phi_\eta(U^{(\eta)})$ whose diameters are in $[c_1a,c_2a[$. Then by Lemma \ref{lm:cluster_decomp} and the Poissonian tail of $\#\mathcal L'$, we have that for all $\epsilon'\in(0,1)$,

\begin{equation}\label{eq:s^2n}
\begin{aligned}
&\mathbb{P}\left[\clu_{A^{(\eta)}(r_{1.5},R_{1.5})}\left(\bls_{<a}^{(\eta)}(U)\right)\ge (1-\epsilon')n\right]\\ =
&\mathbb P\left[\clu_{A(r_{0.5},R_{0.5})}\left(\phi_\eta(\bls_{<a}^{(\eta)}(U))\right) \ge (1-\epsilon')n\right]\\ \le
&\mathbb P\left[\#\{l\in\ls':l\cap A(r_{0.5},R_{0.5})\neq\emptyset\}+\clu_{A(r_{0.5},R_{0.5})}\left(\bls_{<c_1a}(\phi_\eta(U^{(\eta)}))\right)\ge (1-\epsilon')n\right].\\ \le
&\mathbb P\left[\#\{l\in\ls':l\cap A(r_{0.5},R_{0.5})\neq\emptyset\}\ge \epsilon' n\right]+ \mathbb P\left[\clu_{A(r_{0.5},R_{0.5})}\left(\bls_{<c_1a}(\phi_\eta(U^{(\eta)}))\right)\ge (1-2\epsilon')n\right]\\ \le
&\mathbb P\left[\clu_{A(r_{0.5},R_{0.5})}\left(\bls_{<c_1a}(\phi_\eta(U^{(\eta)}))\right)\ge (1-2\epsilon')n\right]+O(s^{2n}).
\end{aligned}
\end{equation}
Moreover, we claim that there exists a constant $c=c(a,r,R,\eta)$ (independent of $U$) such that
\begin{equation}\label{eq:c}
\begin{aligned}
\mathbb P & \left[\clu_{A(r_{0.5},R_{0.5})}
\left(\bls_{<c_1a}(\phi_\eta(U^{(\eta)})\right)\ge n\right]  \le\\
	&\qquad c \cdot\mathbb P\left[\clu_{A(r_{0.5},R_{0.5})}\left(\bls_{<a}(\phi_\eta(U^{(\eta)})))\right)\ge n\right].
\end{aligned}
\end{equation}
In fact, the independence of $\bls_{\ge c_1a}(\phi_\eta(U^{(\eta)}))$ and $\bls_{< c_1a}(\phi_\eta(U^{(\eta)}))$ gives that
\begin{align*}
&\mathbb P[\clu_{A(r_{0.5},R_{0.5})}(\bls_{<a}(\phi_\eta(U^{(\eta)})))\ge n]\\ \ge
&\mathbb P[\clu_{A(r_{0.5},R_{0.5})}(\bls_{<c_1 a}(\phi_\eta(U^{(\eta)})))\ge n,\,\bls_{\ge c_1a}(\phi_\eta(U^{(\eta)}))=\emptyset]\\
=&\mathbb P[\clu_{A(r_{0.5},R_{0.5})}\left(\bls_{<c_1a}(\phi_\eta(U^{(\eta)}))\right)\ge n]\cdot\mathbb P\left[\bls_{\ge c_1a}(\phi_\eta(U^{(\eta)}))=\emptyset\right]\\
\ge&
\mathbb P
[\clu_{A(r_{0.5},R_{0.5})}\left(\bls_{<c_1a}({\phi_\eta(U^{(\eta)})})\right)\ge n]
\cdot
\mathbb P\left[\bls_{\ge c_1a}({\H})=\emptyset\right],
\end{align*}
and \eqref{eq:c} follows by taking $c^{-1}=\mathbb P\left[\bls_{\ge c_1a}({\H})=\emptyset\right]>0$. Therefore by \eqref{eq:s^2n}, \eqref{eq:c} and taking the supremum, we have
\begin{equation}\label{eq:bls_eta}
	\sup_{U\subseteq\H}\mathbb P \left[\clu_{A(r_{1.5},R_{1.5})}\left(\bls^{(\eta)}_{<a}(U)\right)\ge (1-\epsilon')n \right]\le c\cdot\sup_{U\subseteq\H}\mathbb P[\clu_{A(r_{0.5},R_{0.5})}(\bls_{<a}(U))\ge (1-2\epsilon')n]+O(s^{2n}).
\end{equation}

\vspace{2mm}
\noindent\emph{Step 4: 
Compare the crossing cluster number in $A(r_{0.5},R_{0.5})$ and $A(r,R)$.}
\vspace{1mm}
In this step, we will show that for any $\epsilon'\in(0,1)$, there exists $0<q<1$ such that

\begin{equation}\label{eq:exp_decay}
\sup_{U\subseteq\H}\mathbb P\big[\clu_{A(r_{0.5},R_{0.5})}\left(\bls_{<a}(U)\right)\ge (1-2\epsilon')n\big] \le q^n\cdot\sup_{U\subseteq\H}\mathbb P\left[\clu_{A(r,R)}\left(\bls_{<a}(U)\right)\ge (1-3\epsilon')n\right]+O(s^{2n}).
\end{equation}
Recall that $r_\theta$ is increasing in $\theta$, $R_\theta$ is decreasing in $\theta$ (see \eqref{rR}), $r=r_1, R=R_1$, and $A(r,R)$, $A(R_{0.6},R_{0.5}) \subseteq A(r_{0.5}, R_{0.5})$. Therefore we have
\[
	\mathbb P\left[\clu_{A(r_{0.5},R_{0.5})}\left(\bls_{<a}(U)\right)\ge n\right] \le \mathbb P\left[\clu_{A(r,R)}\left(\bls_{<a}(U)\right)\ge n,\clu_{A(R_{0.6},R_{0.5})}\left(\bls_{<a}(U)\right)\ge n\right].
\]
By Lemma \ref{lm:cluster_decomp}, if we write $U':=\{z\in U: |z|<R_{0.8}\}$, we have that
\[
\clu_{A(r,R)}\big(\bls_{<a}(U)\big)\le  \clu_{A(r,R)}\left(\bls_{<a}(U')\right)+\#\{l\in\bls_{<a}(U):l\text{ crosses }A(R,R_{0.8})\},
\]
and
\[
\clu_{A(R_{0.6},R_{0.5})}\big(\bls_{<a}(U)\big) \le \clu_{A(R_{0.6},R_{0.5})}\left(\bls_{<a}(U\backslash U')\right)+\#\{l\in\bls_{<a}(U):l\text{ crosses }A(R_{0.8},R_{0.6})\}.
\]
Combined with the fact that $A(r,R)\cap A(R_{0.6}, R_{0.5}) = \emptyset$, we have 
\begin{align*}
	&\mathbb P\left[\clu_{A(r_{0.5},R_{0.5})}\left(\bls_{<a}(U))\right)\ge (1-2\epsilon')n\right]\\ \le
	& \mathbb P\big[\clu_{A(r,R)}\left(\bls_{<a}(U')\right)\ge (1-3\epsilon') n \text{ and } \clu_{A(R_{0.6},R_{0.5})}\left(\bls_{<a}(U\backslash U')\right)\ge(1-3\epsilon')n\big] \\
	&+\mathbb P\big[\#\{l\in\bls_{<a}(\H):l \text{ crosses } A(R,R_{0.8})\text{ or }A(R_{0.8},R_{0.6})\}\ge \epsilon' n\big]\\
	\le
	&\mathbb P\left[\clu_{A(r,R)}\left(\bls_{<a}(U')\right)\ge(1-3\epsilon')n\right]
\times\mathbb P\left[\clu_{A(R_{0.6},R_{0.5})}\left(\bls_{<a}(U\backslash U')\right)\ge(1-3\epsilon')n\right]+O(s^{2n}),
\end{align*}
where the last inequality follows from the independence of Brownian loop soup in disjoint domains and the super-exponential tail of distribution on the number of loops in $\bls_{<a}(\mathbb H)$ which cross $A(R,R_{0.8})$ or $A(R_{0.8},R_{0.6})$.
Also note that once $\epsilon',\eta$ are fixed, there exists $0<q<1$ (the smaller $\epsilon'$ is, the smaller $q$ is) such that
\begin{equation}
 {\sup_{U\subseteq\H}\mathbb P\left[\clu_{A(R_{0.6},R_{0.5})}\left(\bls_{<a}({U})\right)\ge (1-3\epsilon')n\right]\le q^n}
\end{equation}
 due to BK's inequality \cite{berg} (as in Lemma 9.6 of \cite{sheffield-werner}) for disjoint-occurrence event of a Poissonnian sample. This completes the proof of \eqref{eq:exp_decay}.

\vspace{2mm}
\noindent\emph{Conclusion.}
\vspace{1mm}
To summarize, 
{we deduce \eqref{eq:clus_eta} from 
\eqref{eq:bls_eta} and
\eqref{eq:exp_decay}.
Then}
combining \eqref{eq:prob_decomp}, \eqref{eq:clus_decomp} and \eqref{eq:clus_eta}, we have that for any $\epsilon'\in(0,1)$,
\begin{align*}
f(n+1) &\le \frac{s}{2}f(n) + \mathbb P\left[E_{n,\eta}\right] \\
&\le \frac{s}{2}f(n)+
{\sup_{U\subseteq\H}}
\mathbb{P}\left[\clu_{{ A^{(\eta)}}(r_{1.5},R_{1.5})}\left(\bls^{(\eta)}_{<a}({U})\right)\ge (1-\epsilon')n\right]+O(s^{2n})\\ 
&\le \frac{s}{2}f(n)+ c\cdot\sup_{U\subseteq\H}\mathbb P[\clu_{A(r_{0.5},R_{0.5})}(\bls_{<a}(U))\ge (1-2\epsilon')n]+O(s^{2n})\\
&\le \frac{s}{2}f(n)+ cq^n\cdot
{\sup_{U\subseteq\H}}
\mathbb P[\clu_{A(r,R)}(\bls_{<a}({ U}))\ge(1-3\epsilon')n]+O(s^{2n}),
\end{align*}
which is exactly \eqref{eq:iter} if we take $\epsilon$ to be $3\epsilon'$.

\section{Proof of {Theorem \ref{thm:main}}{}}\label{sec:main_proof}
%In this section, we prove Theorem \ref{thm:main} using the results established in previous sections. 
Denote by $L(r,R)$ the set of all loops in $\H$ crossing $A(r,R)$. Recall that the mass of $L(r,R)$ under the Brownian loop measure is finite, and in the following we denote by $\mu_L$ the Brownian loop measure $\mu$ restricted to $L(r,R)$. 
To deal with single loops,  we abuse the notation  $\cro_A(l)$ to denote the maximum number of non-overlapping time intervals whose image under $l$ cross $A$. In particular, the crossings of a single loop are not necessarily disjoint, and $\cro_A(\{l\})\le\cro_A(l)$, see Figure \ref{pic:difference} for an illustration, and see \eqref{eq:uon} for the reason to define $\cro_A(l)$.
We start with a coarse estimate on the crossing number of an annulus by a single loop in the Brownian loop soup.

\begin{figure}[ht]
\centering
\includegraphics[width=.15\textwidth]{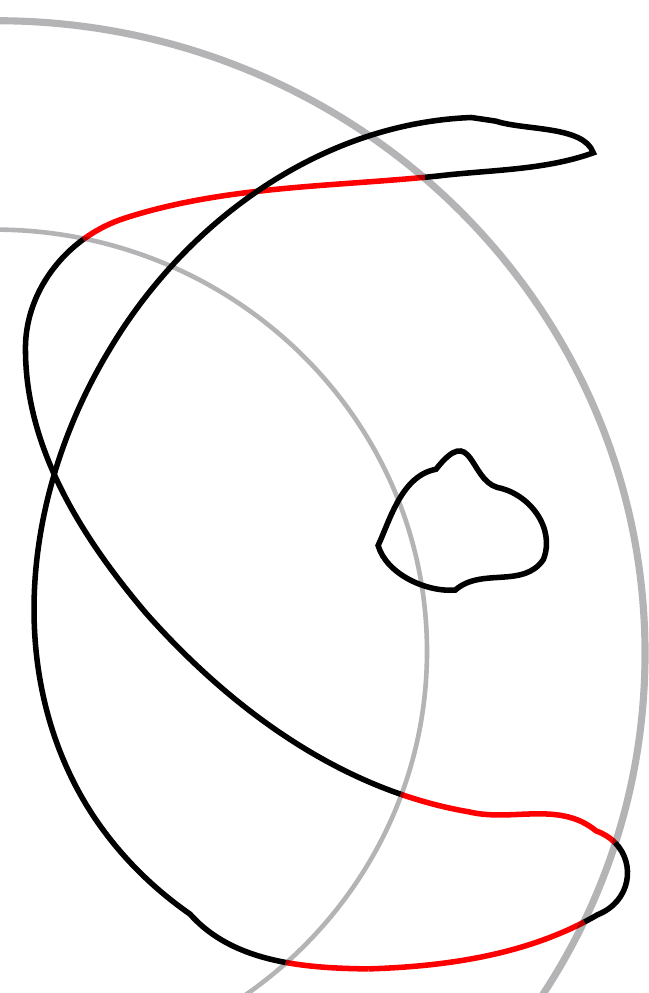}
\caption{For the loop captured from the top-right corner of Figure \ref{pic:cross}, we have $\cro(l)=4$ and $\cro(\{l\})=3$.}
\label{pic:difference}
\end{figure}

\begin{lemma}\label{lm:loop_crossing}
	Let $\bls(\H)$ be the Brownian loop soup with intensity $\lambda\in(0,1]$ on $\H$. Then there exists $q=q(r,R,\lambda)$ such that
	\[
	\mathbb{P}\left[\sum_{l\in\bls(\H)}\cro_{A(r,R)}(l)\ge n\right]= O(q^n).
	\]
\end{lemma}
\begin{proof}
Denote by $\mu^\#_L$ the normalized probability measure on $L(r,R)$ on the trace of an unrooted loop.
For the sake of tracing the loop, we can assume that it takes root inside the annulus $A(R,2R)$ almost surely .

\begin{figure}[ht]
\centering
	\includegraphics[width=0.6\textwidth]{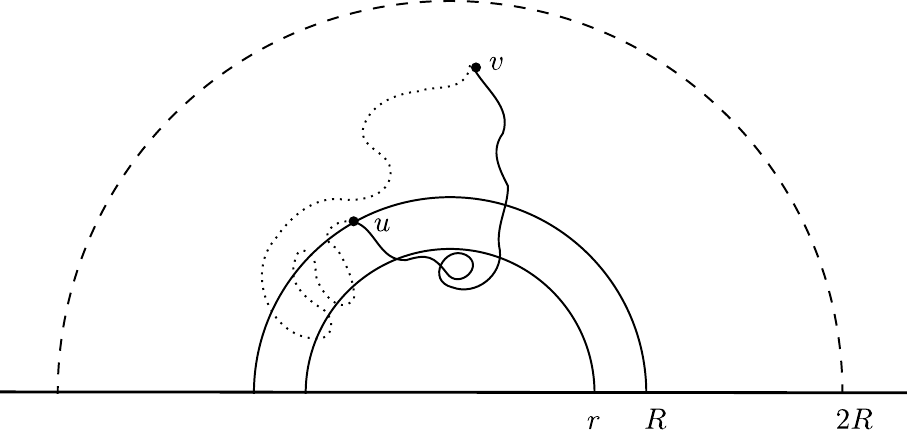}
\caption{Conditioned on the solid line from the root $v$ to some point $u$ on $\partial B_R$, we study the remaining (dotted) path in \eqref{eq:n-2}. In particular, since there are already $2$ crossings on the solid path, we need $n-2$ crossings for the dotted path.}
\label{pic:n-2}
\end{figure}

Conditioned on the trajectory before first returning to $\partial B_R=\{z:|z|= R\}$ after hitting $\partial B_r=\{z:|z|=r\}$, the remaining part is an independent Brownian motion on $\H$ from the landing point on $\partial B_R$ conditioned on coming back to the root, see Figure \ref{pic:n-2}.
Applying  the strong Markov property recursively, we have
\begin{equation}
\label{eq:n-2}
\begin{aligned}
	\mathbb{P}_{\mu^\#_L}[\cro_{A(r,R)}(l)\ge n] \le& \sup_{\substack{u\in \partial B_R\\ v\in A(R,2R) }}\mathbb{P}_{u\rightarrow v}[
	W\text{ crosses }A(r,R)
	\text{ at least }
	n-2 \text{ times }]\\
	\le & \left(\sup_{\substack{u\in \partial B_R\\ v\in A(R,2R)}}\mathbb{P}_{u\rightarrow v}[W\text{ hits }\partial B_r \text{ before returning to } v]\right)^{\lceil\frac{n}{2}-1\rceil}\\
	\le & p^{\frac{n}{2}-1},
\end{aligned}
\end{equation}
where $\mathbb{P}_{u\rightarrow v}$ denotes the normalized (Brownian) interior to interior measure on $\H$ from $u$ to $v$, $W$ is the trajectory under $\mathbb{P}_{u\rightarrow v}$ and
  \[
  p:=\sup_{\substack{u\in \partial B_R\\ v\in A(R,2R)}}\mathbb{P}_{u\rightarrow v}[W\text{ hits }\partial B_r \text{ before returning to } v]<1.
  \]
Then Campbell's second theorem tells that for any $\epsilon>0$,  \begin{align*}
 	&\mathbb{E}\left[\exp\left(-\left(\frac{1}{2}\log p+\epsilon\right)\cdot\sum_{l\in \bls(\H)}\cro_{A(r,R)}(l)\right)\right]\\
 	=&\mathbb{E}\left[\exp\left(-\left(\frac{1}{2}\log p+\epsilon\right)\cdot\sum_{l\in \bls(\H)\cap L(r,R)}\cro_{A(r,R)}(l)\right)\right]\\
 =& \exp\left( -\int_{L(r,R)}\left[1-\exp\left(-\left(\frac{1}{2}\log p+\epsilon\right)\cdot\cro_{A(r,R)}(l)\right)\right]d\mu(l) \right)\\
 \le &\exp\left(|\mu_L|\cdot\mathbb{E}_{\mu_L^\#}\left[\exp\left(-\left(\frac{1}{2}\log p+\epsilon\right)\cdot\cro_{A(r,R)}(l)\right)\right] \right)\\
 \le & p^{-1}\cdot\exp\left(|\mu_L|/(1-e^{-\epsilon}) \right).
\end{align*}
This implies that
\[
\mathbb{P}\left[\sum_{l\in\bls(\H)}\cro_{A(r,R)}(l)\ge n\right]=\exp\left(|\mu_L|/(1-e^{-\epsilon}) \right)p^{\frac{n}{2}-1}e^{n\epsilon}.
\]
Then Lemma \ref{lm:loop_crossing} follows by taking $\epsilon$ sufficiently small such that $q=\sqrt{p}e^\epsilon<1$.
\end{proof}

Further, we show in the next lemma that, the probability on the total crossings of single loops in $\bls(\Omega)$ also has super-exponential decay. Notice that $\cro_{A(r,R)}(\bls(\Omega))\ne\cro_{A(r,R)}(\cle(\Omega))$, because for $\cro_{A(r,R)}(\bls(\Omega))$ we only count crossings formed by single loops, not clusters.

\begin{proposition}\label{prop:loop_crossing}
Let $\bls(\Omega)$ be a Brownian loop soup with intensity $\lambda\in(0,1]$ inside a simply connected subdomain $\Omega\subseteq\H$, then
\[
\sup_{\Omega\subseteq\H}{\mathbb{P}[\cro_{A(r,R)}(\bls(\Omega))\ge n]} \text{ decays super-exponentially.}
\]
\end{proposition}

\begin{proof}
By monotonicity of the crossing number \eqref{eq:property_cross},
it suffices to show that
\[
{\mathbb{P}[\cro_{A(r,R)}(\bls(\H))\ge n]} \text{ decays super-exponentially.}
\]

We can decompose the traces of each loop in $L(r,R)$ into pieces of crossings (from $\partial B_r$ to $\partial B_R$ or from $\partial B_R$ to $\partial B_r$) and Brownian excursions connecting consecutive crossings by Ito’s excursion theory \cite{pitman-yor}. For each crossing, conditioned on its starting point and end point, it is distributed according to the normalized Brownian excursion measure in $A(r,R)$ independent of other parts of the loop. For the purpose of estimating $\cro_{A(r,R)}(\bls(\H))$, by summing over the number of crossings (not necessarily disjoint) of loops in the Brownian loop soup the upper bounds in \ref{lm:loop_crossing} and then selecting $n$ disjoint crossings out of them, we have
\begin{equation}\label{eq:uon}
\begin{aligned}
	\mathbb{P}[\cro_{A(r,R)}(\bls(\H))\ge n] &\le \sum_{k\ge n}\mathbb{P}[\sum_{l\in\bls(\H)}\cro_{A(r,R)}(l)=k]\cdot \binom{k}{n}\cdot u_n(r,R)\\
	&\le C\cdot u_n(r,R)\cdot\sum_{k\ge n}q^k\cdot \binom{k}{n},
\end{aligned}
\end{equation}
where $u_n(r,R):=\sup\limits_{\substack{x_1,\ldots,x_n\in \partial B_r\\y_1,\ldots,y_n\in \partial B_R}}\mathbb{P}[\text{Brownian excursions from }x_1,\ldots,x_n \text{ to }y_1,\ldots,y_n$ inside $A(r,R)$ are disjoint$]$ and by Lemma \ref{lm:loop_crossing}, there exists $C>0$ and $q<1$ such that 
\[
\mathbb{P}\left[\sum_{l\in\bls(\H)}\cro_{A(r,R)}(l)=k\right]\le C\cdot q^k.
\]
We first look at the factor  $v_n:=\sum_{k=n}^{\infty}q^k\cdot\binom{k}{n}$ in \eqref{eq:uon}. In fact,
\begin{align*}
	(1-q)v_n &= \sum_{k=n}^{\infty} q^k\cdot\binom{k}{n}-\sum_{k=n}^{\infty} q^{k+1}\cdot \binom{k}{n}=q^n + \sum_{k=n+1}^{\infty}q^k\left(\binom{k}{n}-\binom{k-1}{n}\right)
	%=q^n+\sum_{k=n+1}^{\infty}q^k\binom{k-1}{n-1}
	=qv_{n-1},
\end{align*}
i.e. $v_n$ grows exponentially with exponent $\frac{q}{1-q}$. Therefore to prove the desired super-exponential decay for \eqref{eq:uon}, it suffices to prove that $u_n(r,R)$ decays super-exponentially. To this end,	one can apply the Fomin's identity (for example, see \cite{kozdron-lawler}) for the non-intersection probability of a random walk excursion and loop-erased random walks (which is obviously larger than the non-intersection probability of random walk excursions). By conformal invariance of the Brownian excursion, we choose a conformal map $\varphi:A(r,R)\cap\H \to\mathbb D$
such that $\varphi(\partial B_R)=\{e^{i\theta}:\theta\in]-\theta_1,\theta_1[\}$ and $\varphi(\partial B_r)=\{e^{i\theta}:\theta\in]-\theta_2+\pi,\theta_2+\pi[\}$ for some $\theta_1+\theta_2<\pi$. Then
    \begin{align*}
    u_n(r,R)
    &\le \sup_{\substack{1\le k\le n, x_k\in ]-\theta_1,\theta_1[,\\ y_k\in ]-\theta_2+\pi,\theta_2+\pi[}}\det\left[\frac{1-\cos(x_j-y_j)}{1-\cos(x_j-y_l)} \right]_{\substack{1\le j,l \le n}}\\
    & \le 2\sup_{\substack{1\le k\le n, x_k\in ]-\theta_1,\theta_1[,\\ y_k\in ]-\theta_2+\pi,\theta_2+\pi[}}\det\left[\frac{1}{1-\cos(x_j-y_l)} \right]_{\substack{1\le j,l\le n}}.
    \end{align*} 
Among the choice $x_j$, $1\le j\le n$, there exist a pair of indices $i_1\neq i_2$ such that $|x_{i_1}-x_{i_2}|\le \frac{2\pi}{n}$. By subtracting the $i_1$-th row from the $i_2$-th row,  the $i_2$-th row is the vector 
    \[
    \left[\frac{\cos(x_{i_2}-y_l)-\cos(x_{i_1}-y_l)}{(1-\cos(x_{i_1}-y_l))(1-
    \cos(x_{i_2}-y_l))}\right]_{1\le l\le n},
    \]
    whose modulus ($L^2$-norm) is less than
    $\frac{2\pi}{\sqrt{n}(1-\cos(\pi-\theta_1-\theta_2))^2}.$
    By performing the same procedure on the remaining $n-1$ rows, we have
    \[
    u_n(r,R) \le \left(\frac{4\pi}{(1+\cos(\theta_1+\theta_2))^2}\right)^{n-1}\cdot(n!)^{-\frac{1}{2}},
    \]
    which implies that $u_n(r,R)$ decays super-exponentially fast. The conclusion then follows by \eqref{eq:uon}.
\end{proof}

\begin{proof}[Proof of Theorem \ref{thm:main}]
By \eqref{eq:cro-comp} and \eqref{eq:CLEtoBLS}, it suffices to show that for all $s\in(0,1)$,
\begin{equation*}
\sup_{\Omega\subseteq\H}\mathbb P\left[\com_{A(r,R)}(\mathcal B(\Omega))\ge n\right]=O(s^n).
\end{equation*}

Introduce $a:=(R-r)/8$ to divide the Brownian loop soup into two parts according to their diameters, then by Lemma \ref{lm:decomposition},
\begin{align*}
\com_{A(r,R)}(\bls(\Omega))\le  &\com_{A(r+a,R-a)}(\bls_{<a}(\Omega))+\#\{l\in\bls_{\ge a}(\Omega):l\subset A(r,R))\} \notag \\
    &+\cro_{A(r,r+a)}(\bls_{\ge a}(\Omega))+\cro_{A(R-a,R)}(\bls_{\ge a}(\Omega)).
\end{align*}
Besides, Lemma \ref{lm:CompToCluster} implies that
\begin{equation*}
\begin{aligned}
\com_{A(r+a,R-a)}(\bls_{<a}(\Omega))&\le \clu_{A(r+2a,R-2a)}(\bls_{<a}(A(r+a,R-a)\cap\Omega)).
\end{aligned}
\end{equation*}
Then the conclusion follows by combining Proposition \ref{prop:clusterMajor}, Proposition \ref{prop:loop_crossing} and the Poisson tail of $\#\{l\in\bls_{\ge a}(\H):l\subset A(r,R))\}$.
\end{proof}

\section{Proof of {Corollary \ref{cor:main2}}{}}\label{sec:general_ann}
In this section, we prove Corollary \ref{cor:main2}, which generalizes Theorem \ref{thm:main} to the crossing estimates of arbitrary quads, with the same spirit as in \cite{kemppainen-smirnov}. By conformal invariance of CLEs, without loss of generality, we will assume in the whole section that $\Omega=\H$. First, let us extend the crossing estimated in Theorem \ref{thm:main} to hold for inner annuli uniformly on their mudulus.

\begin{lemma}\label{lm:cro_annuli}
\label{lm:cro(H)}
Given a non-nested simple $\cle_\kappa(\H)$, $\kappa\in(\frac{8}{3},4]$, we have that for all $s\in(0,1)$, $z_0\in\mathbb C$ and $0<r<R$,
\[
\mathbb P\left[\cro_{A_{z_0}(r,R)}(\cle_\kappa(\H))\ge n\right] = O(s^n)
\]
where the constant in $O(s^n)$ depends on $\kappa$ and $R/r$.
\end{lemma}

\begin{proof}
It readily follows from Theorem \ref{thm:main} that the result holds for $\Im z_0\le 0$. If $\Im z_0>0$, by the Brownian loop-soup construction of CLEs and the conformal invariance of Brownian loop soup on $\H$, it suffices to prove that for $\lambda=(3\kappa-8)(6-\kappa)/2\kappa$ and for all $y\ge 0$, $0<r<1$ and $s\in(0,1)$, 
\begin{equation}\label{eq:com_iy}
	\mathbb P\left[\com_{A_{iy}(r,1)}(\bls^\lambda(\H))\ge n\right]=O(s^n).
\end{equation}
For each $y>2$ and $\epsilon$ sufficiently small, it holds by Lemma \ref{lm:decomposition} that 
	\begin{align*}
		&\mathbb P\left[\com_{A_{iy}(r,1)}(\bls^\lambda(\H))\ge n\right]\\ \le
		&\mathbb P\left[\com_{A_{iy}\left(\frac{3r+1}{4},\frac{r+3}{4}\right)}(\bls^\lambda(\H+i(y-2)))\ge (1-2\epsilon)n\right]\\
		&+\mathbb{P}\left[\cro_{A_{iy}\left(r,\frac{3r+1}{4}\right)}\left(\bls^\lambda(\H+i(y-2))^\bot\right)\ge\epsilon n\right]\\
		 &+\mathbb{P}\left[\cro_{A_{iy}\left(\frac{r+3}{4},1\right)}\left(\bls^\lambda(\H+i(y-2))^\bot\right)\ge\epsilon n\right] \\
		\le &\mathbb P\left[\com_{A_{2i}\left(\frac{3r+1}{4},\frac{r+3}{4}\right)}(\bls^\lambda(\H)) \ge (1-2\epsilon)n\right]+O(s^n),
	\end{align*}
by shifting $\H+i(y-2)$ downwards by the distance $i(y-2)$, where the term $O(s^n)$ follows from Proposition \ref{prop:loop_crossing} because any crossing arc of $A_{iy}\left(\frac{r+3}{4},1\right)$ (or $A_{iy}\left(r,\frac{3r+1}{4}\right)$) must intersect both $\mathbb{R}+i(y-2)$ and $A_{iy}\left(\frac{r+3}{4},1\right)$, and these arcs are bound to cross one of the annuli in the left picture of Figure \ref{pic:covering}. Similarly, the probability  of the event $\{\com_{A_{2i}\left(\frac{3r+1}{4},\frac{r+3}{4}\right)}$ $(\bls^\lambda(\H))\ge n\}$
can be bounded by the probability of a union crossing events of annuli centered at the origin, see the left picture of Figure \ref{pic:covering}, which completes the proof of \eqref{eq:com_iy} for $y>2$.

\begin{figure}[ht]
\centering
	\includegraphics[width=0.4\textwidth]{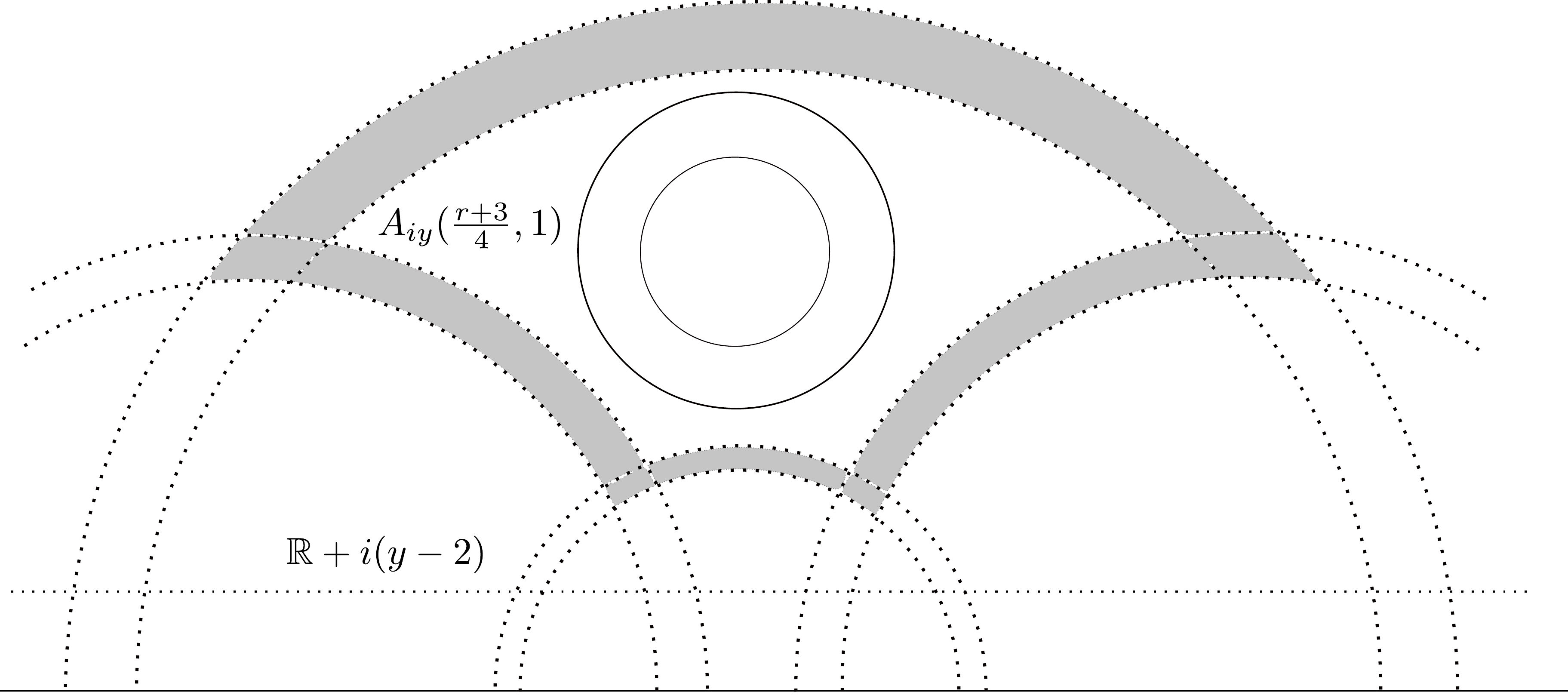}
	\includegraphics[width=0.4\textwidth]{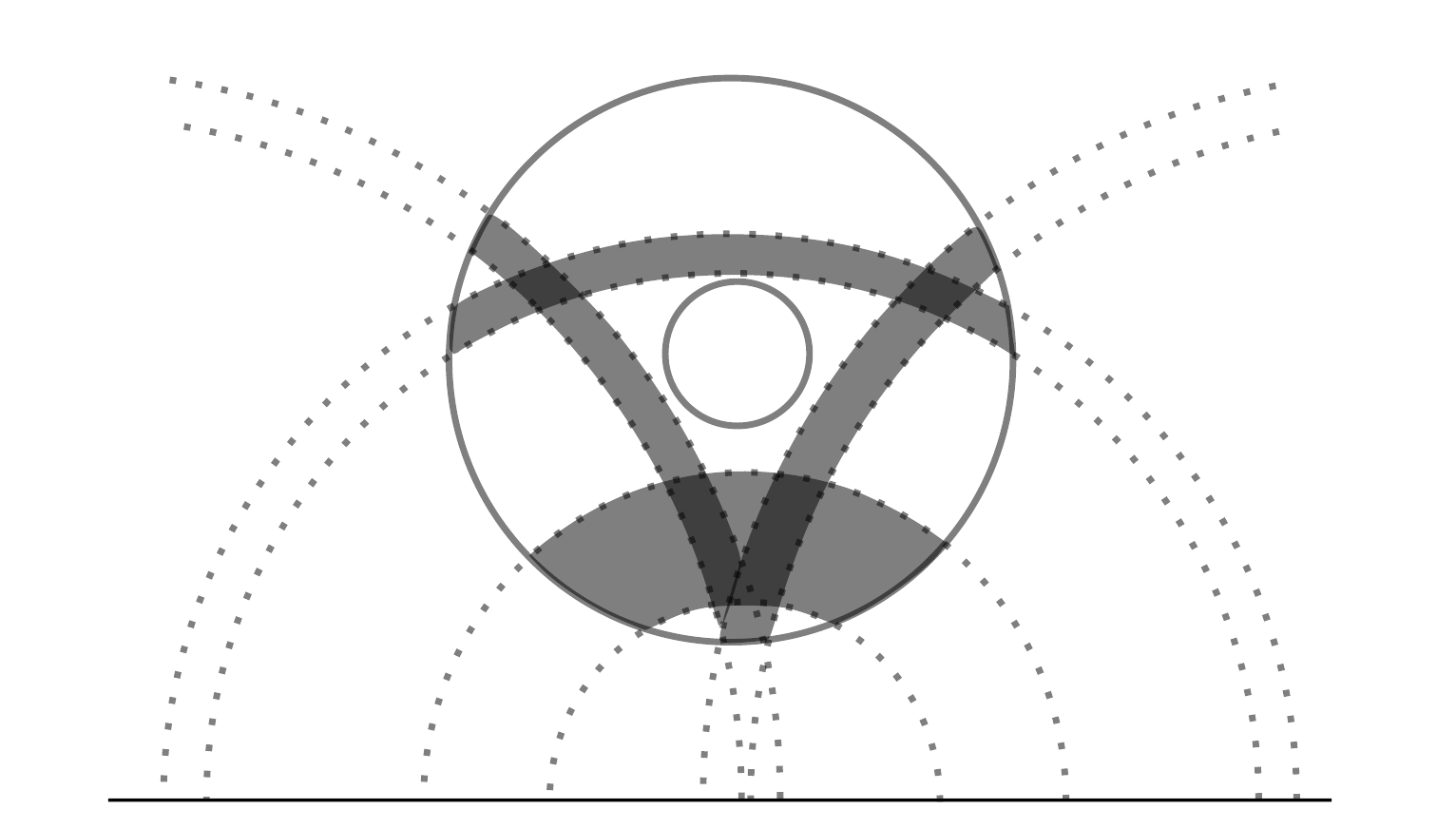}
\caption{Each crossing is bound to cross one of the shaded annulus sectors.}
\label{pic:covering}
\end{figure}

For $y\in[0,2]$, we are going to establish \eqref{eq:com_iy} uniformly in $y$ by finding a finite number of annuli $A_1,\ldots,A_k$ such that for any $A_{iy}(r,1)$, there exists at least one $A_j\subseteq A_{iy}(r,1)$, $j=1,\ldots,k$, therefore it is not hard to see that
\[
\mathbb P\left[\com_{A_{iy}(r,1)}(\bls^\lambda(\H))\ge n\right]\le\max_{1\le j\le k}\mathbb P\left[\com_{A_j}(\bls^\lambda(\H))\ge n\right]=O(s^n),\ y\in[0,2],
\] 
where the $O(s^n)$ term for $j=1,\ldots,k$ can be bounded similarly by the probability of a union of crossing events as illustrated in the right picture of Figure \ref{pic:covering}.
Effectively, if we choose $y_j:=(j-1)\cdot\frac{1-r}{2}$ for $j=1,\ldots,k$, where $k=\lceil\frac{4}{1-r}\rceil+1$, then	
	\[
	A_{iy_j}(r,\frac{r+1}{2}) \subseteq A_{iy}(r,1) \text{ for all } y\in[y_{j-1},y_j].
	\] 
This completes the proof of \eqref{eq:com_iy}.
\end{proof}

The proof of Corollary \ref{cor:main2} for generic quads $Q=(V;S_k,k=0,1,2,3)$ with $S_1, S_3\subset\mathbb{R}$ proceeds by connecting $S_1$ and $S_3$ by a chain of annuli of fixed radii ratio, for which the number of annuli needed depends only on $m(Q)$. To analyze $m(Q)$, we need the concept of the extremal length, which also gives the conformal modulus. Let $\Gamma$ be a family of locally rectifiable curves in an open set $D$ in the complex plane. If $\rho:D\rightarrow[0,\infty]$ is square-integrable on $D$, then define 
\[
A_\rho (D)  =  \iint_D \rho^2 (z)d^2z \ \ 
\textrm{ and } \  L _  \rho  ( \Gamma )  =  \inf _ {\gamma \in \Gamma }
 \int_\gamma  \rho(z) |dz|,
\]
where $d^2z$ denotes the Lebesgue measure on the complex plane and $|dz|$ denotes the Euclidean element of length. Then the extremal length of $\Gamma$ is defined by
\[
m(\Gamma) := \sup_{\rho \in P} \frac{L_\rho(\Gamma)^2}{A_\rho(D)}.
\]
From the definition it is clear that the extremal length satisfies a simple monotonicity property: if $\Gamma_1\subseteq \Gamma_2$, then $m(\Gamma_1)\ge m(\Gamma_2)$. Moreover, it also agrees with the conformal modulus $m(Q)$ we introduced in Section \ref{sec:1.2} as the unique number for which $Q$ can be conformally onto a rectangle $[0,1]\times[0,m(Q)]$ with $S_k$ mapped to the four sides of the rectangle and $S_0$ mapped to $[0,1]\times\{0\}$, i.e. (cf. eg.  \cite{ahlfors})
\[
m(\Gamma)=m(Q),
\]
where $\Gamma$ is the family of all curves joining $S_0$ and $S_2$ inside $Q=(V;S_k,k=0,1,2,3)$.

We begin with an estimate on the extremal length following \cite[pages 719-720]{kemppainen-smirnov}. 
\begin{lemma}\label{lem:extre_len}
	Suppose that $Q=(V;S_k,k=0,1,2,3)$ has conformal modulus $m(Q)\ge 36$. Then there exist $z_0\in\mathbb{C}$ and $r>0$ such that any curve connecting $S_0$ and $S_2$ inside $V$ must cross an annulus $A_{z_0}(r,2r)$.
\end{lemma}
\begin{proof}
%[Proof (see \cite{kemppainen-smirnov})]
% Let $\Gamma^*$ be the dual family of curves joining
%∂1Q to ∂3Q inside Q, then m(Γ) = 1/m(Γ∗).
Let 
\[d_1=\inf\{\text{length}(\gamma):\gamma \text{ joining }S_1,S_3 \text{ inside }V\}
\]
be the distance between $S_1$ and $S_3$ in the inner Euclidean metric of $Q$, and let $\gamma^*$ be a curve of length $\le 2d_1$ joining $S_1$ and $S_3$ inside $V$. We are going to show that any crossing $\gamma$ (joining $S_0$ and $S_2$ inside $V$) of $Q$ has diameter $d\ge 4d_1$. Indeed, working with the extremal length of the dual family of curves
\[
\Gamma^*=\{\gamma^*:\gamma^* \text{ connects } S_1 \text{ and }S_3 \text{ inside }V\},
\]
 take a metric $\rho$ equal to $1$ in the $d_1$-neighborhood of $\gamma$ and zero outside the $d_1$-neighborhood of $\gamma$. Then its area integral is at most $(d+2d_1)^2$, and any $\gamma^*\in\Gamma^*$ has length at least $d_1$ since $\gamma\cap\gamma^*\neq\emptyset$ must run through the support of $\rho$ for a length of at least $d_1$. Therefore $1/m(Q) = m(\Gamma^*) \ge d^2_1/(d + 2d_1)^2$, hence
\[
d \ge (\sqrt{m(Q)}-2)d_1 \ge 4d_1. 
\]

Now if we take an annulus $A$ centered at the middle point of $\gamma^*$ with inner radius $d_1$
and outer radius $2d_1$, every crossing $\gamma$ of Q contains
a crossing of $A$ because $\gamma$ has to intersect $\gamma^*$, which is contained inside the inner circle of $A$, and $\gamma$ has to intersect the outer circle of $A$ if its diameter is larger than $4d_1$.
\end{proof}

{
\begin{proof}[Proof of Corollary \ref{cor:main2}]

Let us decompose the set of crossings curves (from $S_0$ to $
S_2$ or from $S_2$ to $S_0$ inside $V$) of the quad $Q=(V;S_k,k=0,1,2,3)$.
In fact, if we map conformally $Q$ onto a rectangle $[0,1]\times[0,m(Q)]$ by $\phi_Q$, we can choose $K>0$ large enough, which depends only on $m(Q)$, such that for any $0\le i,j \le K-1$, the set of curves $\Gamma_{i,j}$ connecting $[\frac{i}{K},\frac{i+1}{K}]\times\{0\}$ and $[\frac{j}{K},\frac{j+1}{K}]\times\{m(Q)\}$ 
inside $\Omega$ has extremal length larger than $36$.
Then by Lemma \ref{lem:extre_len}, any curve in $\phi_Q^{-1}(\Gamma_{i,j})$ has to cross an annulus $A_{z_{i,j}}(r_{i,j},2r_{i,j})$ for some $z_{i,j}\in\mathbb C$ and $r_{i,j}>0$. In other words, any curve crossing $Q$ has to cross one of the $K^2$ annuli $(A_{z_{i,j}}(r_{i,j},2r_{i,j}))_{0\le i,j\le K-1}$.

Therefore, our crossing event is included in the union of events \[\{\com_{A_{z_{i,j}}(r_{i,j},2r_{i,j})}(\cle_\kappa(\H))>n/K^2\},\]
and we can finish the proof {by Lemma \ref{lm:cro_annuli}}.
\end{proof}
}

\section{Proof of {Corollary \ref{cor:complexity}}{}}\label{sec:complexity}
Let us now illustrate why our result implies the assumption of \cite[Corollary 1.7]{basok-chelkak}. 
Let $\Omega$ be a planar simply-connected domain and $\lambda_1,\ldots,\lambda_N\in\Omega$ be a collection of pairwise distinct punctures in $\Omega$. Given a loop ensemble in $\Omega\setminus\{\lambda_1,\ldots,\lambda_N\}$, we delete all loops surrounding zero or one puncture, and consider the collection of homotopy classes of loops that surround at least two punctures, which is called a \emph{macroscopic lamination.}
We are interested in the \emph{complexity} $|\Gamma|_{\mathcal T_\Omega}$ of a macroscopic lamination for a fixed triangulation $\mathcal{T}_\Omega=(\{\lambda_1,\ldots,\lambda_N,\partial\Omega, \mathcal E,\mathcal F\})$ of $\Omega\setminus\{\lambda_1,\ldots,\lambda_N\}$ whose $N+1$ vertices are $\lambda_1,\ldots,\lambda_N$ and the boundary of $\Omega$. 
Roughly speaking, $|\Gamma|_{\mathcal T_\Omega}$ is the minimal possible (in the free homotopy class) number of intersections of loops in $\Gamma$ with the edges of $\mathcal{T}_\Omega$. We refer interested readers to \cite{basok-chelkak} for detailed discussions and pictures therein.
The definition of the complexity depends on the choice of the triangulation $\mathcal{T}_\Omega$, but for each two such choices, the complexities differ by no more than a multiplicative factor independent of $\Gamma$. 
For a fixed triangulation $\mathcal T_\Omega$ of $\Omega\setminus\{\lambda_1,\ldots,\lambda_N\}$, the laminations on $\Omega\setminus\{\lambda_1,\ldots,\lambda_N\}$ are parametrized by multi-indices $\textbf n=(n_e)\in\mathbb{N}^\mathcal{E}$ (satisfying certain conditions), where $n_e:=\#\{\Gamma\cap e\}$. Then the complexity $|\Gamma|_{\mathcal T_\Omega}$ (with respect to triangulation $\mathcal T_\Omega$) can be expressed as
 \[
 |\Gamma|_{\mathcal T_\Omega}=\min_{\Gamma':\ \Gamma' \text{ is homotopic to } \Gamma}\#\{\Gamma'\cap\mathcal T_\Omega\},
 \]
 where $\#\{\Gamma'\cap\mathcal T_\Omega\}$ denotes the number of intersections of all loops in $\Gamma'$ with edges of $\mathcal T_\Omega$.
 
 We can assume by the conformal invariance of CLEs that $\Omega=\H$ and $|\lambda_1|<|\lambda_2|<\ldots<|\lambda_N|$ up to a re-ordering of punctures. We choose a triangulation $\mathcal T_\H$ of $\H\setminus\{\lambda_1,\ldots,\lambda_N\}$ such that for any $i<j$, each edge of $\mathcal T_\H$ connecting $\lambda_i,\ \lambda_j$  is a path between $\lambda_i$ and $\lambda_j$ inside $A(|\lambda_i|,|\lambda_j|)$, and any edge between a puncture $\lambda_i$ and $\partial\Omega$ is an arc of $\partial B_{\lambda_i}$. It is not hard to see that the complexity of any macroscopic lamination is bounded by the sum of crossings up to a multiplicative constant.  
 %Nevertheless considering the nesting feature, we are interested in the crossing estimates of CLEs in simply connected domains. 
 
\begin{figure}[ht]
\centering
\includegraphics[width=0.5\textwidth]{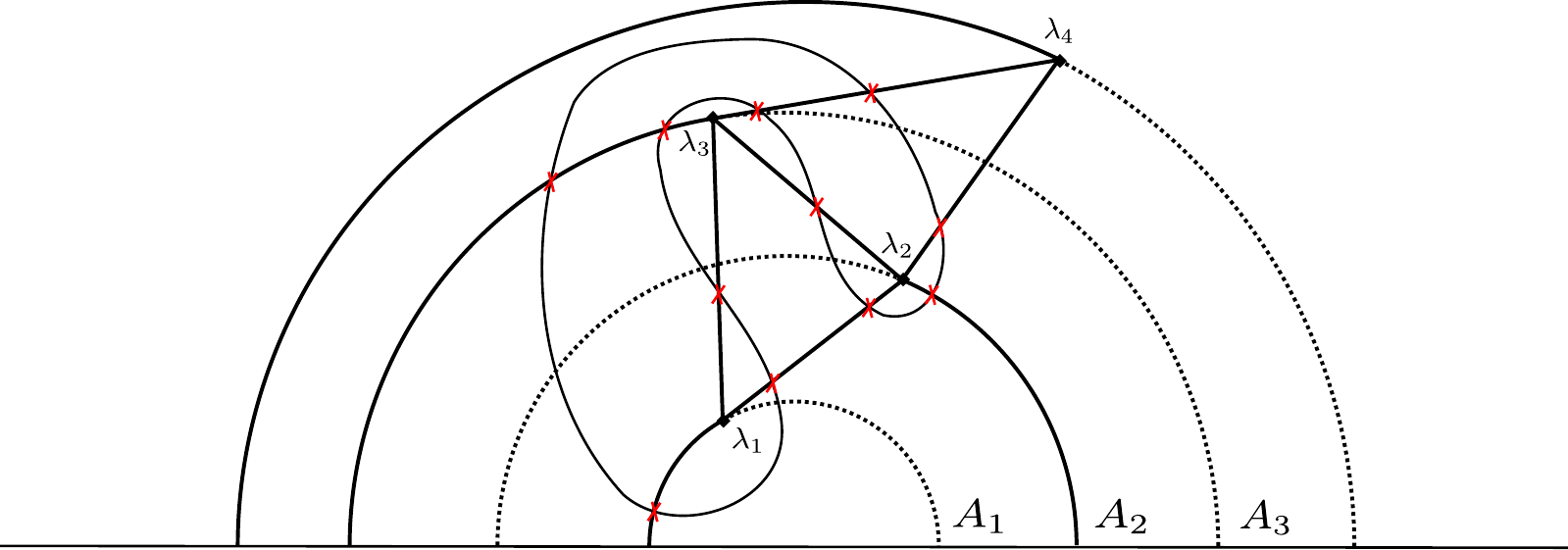}
\caption{
An illustration of the triangulation we adopt and the complexity of a loop. Note that the cross (intersection with segment $(\lambda_3\lambda_4)$) inside $A_3$ corresponds to a crossing of $A_2$ in the proof of Lemma \ref{lm:complexity}}
\end{figure}
 
\begin{lemma}\label{lm:complexity}
Let $\Omega'$ be a simply connected subdomain of $\H$. For each macroscopic lamination $\Gamma$ in $\Omega'$, we have
\[
|\Gamma|_{\mathcal{T}_{\H}}\le 6(N-1)\sum_{i=1}^{N-1}\cro_{A_i}(\Gamma).
\]
\end{lemma}

\begin{proof}
Suppose that 
\[
\Gamma'\in\operatornamewithlimits{argmin}_{\tilde{\Gamma}\text{: homotopic to $\Gamma$}}\#\{\tilde{\Gamma}\cap\mathcal T_\H\} ,
\]
such that
\[
|\Gamma|_{\mathcal T_\H}=\sum_{e\in\mathcal T_\H}\#\{\Gamma'\cap e\} \text{ and }\cro_{A_i}(\Gamma')\le\cro_{A_i}(\Gamma) \text{ for each }i\le N-1.
\]
For any $e\in\mathcal T_{\H}$ and $x\in\Gamma'\cap e$, denote by $l_x$ the loop in $\Gamma'$ that $x$ belongs to. Suppose that $l_x$ is rooted at $x$ and $l_x$ is parametrized by $\mathbb{R}$, denote by
\[
t_-:=\inf\{t\ge 0:l_x(-t)\in\cup_{i=1}^N \partial B_{|\lambda_i|}\}
\]
and
\[
 t_+:=\inf\{t\ge 0:l_x(t)\in\cup_{i=1}^N \partial B_{|\lambda_i|},\ |l_x(t)|\neq|l_x(t_-)|\}.
\]
It is not hard to see that $t_+$ exists and by the minimality of $\Gamma'$, there is at most one another intersection (if $x$ lies on one of the arcs $\partial B_{\lambda_i}$) of $l_x((-t_-,t_+])$ and $e$ except $x$. Therefore there exists $i\le N-1$ such that $l_x((-t_-,t_+])$ crosses $A_i$. By summing over all possibilities of $i$, we have
\[\#(l_x\cap e)\le 2\sum_{i=1}^{N-1}\cro_{A_i}(l_x).\]
Since $\mathcal T_\H$ has $3(N-1)$ edges, we further get
\[\sum_{e\in\mathcal T_\H}\#(l_x\cap e)\le 6(N-1)\sum_{i=1}^{N-1}\cro_{A_i}(l_x).\]
Sum over all the loops in $\Gamma$, notice that they are disjoint by definition, then
\[|\mathcal T_\H|=\sum_{e\in\mathcal T_\H}\#(\Gamma'\cap e)\le 6(N-1)\sum_{i=1}^{N-1}\cro_{A_i}(\Gamma')\le 6(N-1)\sum_{i=1}^{N-1}\cro_{A_i}(\Gamma).\]
\end{proof}
Using Lemma \ref{lm:complexity} and Theorem \ref{thm:main}, we obtain without difficulty the following super-exponential decay of the probability of the complexity. 
\begin{corollary}
    For any simply connected subdomain of $\H$, let $\cle_\kappa(\Omega)$ be a non-nested conformal loop ensemblle with $\kappa\in(\frac{8}{3},4]$ in $\Omega$. Then for any $s>0$,
    \[
    \sup_{\Omega\subseteq\H}\prob{|\cle_\kappa(\Omega)|_{\mathcal{T}_{\H}}>n}=O(s^n).
    \]
\end{corollary}
Now we are ready to conclude the main application of Theorem \ref{thm:main}. We will add superscripts to distinguish non-nested $\cle^\text{n-nested}_\kappa$ and nested $\cle^\text{nested}_\kappa$.
\begin{reptheorem}{cor:complexity}
 Let $\Theta_\Omega$ be a random sample of the nested $\cle_\kappa$, $\frac{8}{3}<\kappa\le4$, in $\Omega$ and let $\Theta^\delta_\Omega$ be the double-dimer loop ensemble on a Temperlean discretization $\Omega^\delta\subset\delta\mathbb{Z}^2$ of $\Omega$. Denote by $\Theta\sim\Gamma$ the event that the macroscopic lamination of a loop ensemble $\Theta$ is $\Gamma$. Then
\[
\mathbb{P}_{\cle^\text{nested}_\kappa}[\Theta_\Omega\sim\Gamma]=O(R^{-|\Gamma|}) \text{  as }|\Gamma| \rightarrow\infty \text{  for all }R>0.
\]
Therefore by \cite[Corollary 1.7]{basok-chelkak},  $\mathbb{P}_{\text{double-dimer}}[\Theta^\delta_\Omega\sim\Gamma]\to\mathbb{P}_{\cle^\text{nested}_4}[\Theta_\Omega\sim\Gamma]$ as $\delta\to 0$ for all macroscopic laminations $\Gamma$. 
\end{reptheorem}

\begin{proof}
% For any $1\le p,q\le N$, denote by $\Gamma_{(p,q)}$ the set of loops in the nested $\cle_\kappa$ surrounding $\lambda_p,\lambda_q$. On one hand, due to \cite[Lemma 21]{dubedat}, there exists $c>0$ such that 
% \[\mathbb P[{|\Gamma_{(p,q)}|}\ge n]\le \exp(-cn^{3/2}).\] 
We can upper-bound the complexity of the nested $\cle_\kappa$ by looking separately at the collection of loops $\Gamma_\Lambda=\{\gamma_1,\ldots,\gamma_N\}$ surrounding  the same subset $\Lambda\subseteq\{\lambda_1,\ldots,\lambda_N\}$ which contains at least two punctures. In addition, we order loops in $\Gamma_\Lambda$ such that $\gamma_{i+1}$ lies inside $\gamma_i$. By abusing the notation slightly we denote by $\mathbb{P}[|\gamma_1|=n_1,\ldots,|\gamma_i|=n_i]$ the quantity 
\[
\sup_{\Omega\subseteq \H}\mathbb{P}[ \gamma_1,\ldots,\gamma_i\in \cle^{\text{n-nested}}_\kappa(\Omega): \gamma_1,\ldots,\gamma_i\text{ encircles } \Lambda \text{ and } |\gamma_1|_{\mathcal{T}_\H}=n_1,\ldots,|\gamma_i|_{\mathcal{T}_\H}=n_i].
\]
Note that the loops in $\Gamma_\Lambda$ are homotopic to each other since they do not intersect. In particular, their complexities coincide. Therefore 
\[
\mathbb{P}[|\gamma_1|=n_1,\cdots,|\gamma_j|=n_j] \text{ is non-zero only if } n_1=\ldots=n_j.
\]
Using independence of the loop ensemble inside $\gamma_{\lfloor j/2 \rfloor}$, for any $C>0$, we have
% On the other hand, by the nesting structure of CLEs, the crossing number of the $j$-th loop in $\Gamma_{(p,q)}$ (of any crossing-quad) is less than the crossing number of the non-nested CLE inside the $(j-1)$-th loop in $\Gamma_{(p,q)}$. If we denote by $\Gamma^k_{(p,q)}$ the interior of the $k$-th loop in $\Gamma_{(p,q)}$ and sum over all loops surrounding $\lambda_p,\lambda_q$, for any $\Lambda>0$, we have

\begin{align*}
    &\prob{|\gamma_1|=\ldots=|\gamma_j|=n}\cdot e^{Cjn}\\ \le
    &\prob{|\gamma_1|>0,\ldots,|\gamma_{\lfloor j/2 \rfloor}|>0}\cdot \prod_{i=\lfloor j/2 \rfloor+1}^j \left(\prob{|\gamma_i|=n}\cdot e^{2Cn}\right)\\ \le
    & \exp(-c(j/2)^{3/2})\cdot\prod_{i=\lfloor j/2 \rfloor+1}^j \left(\prob{|\gamma_i|=n}\cdot e^{2Cn}\right),
\end{align*}
where the exponential term is due to \cite[Lemma 21]{dubedat} on the tail of the distribution of the number of loops surrounding two points.

Because the complexity of $\gamma_{i+1}$ is less than the complexity of the non-nested $\cle_\kappa$ inside $\gamma_i$, this implies that
\begin{align*}
\mathbb{E}\left[\exp\left( C\cdot |\Gamma_\Lambda|_{\mathcal{T}_\H}) \right)\right]
\le \sum_{j\ge 0} e^{-c(j/2)^{3/2}} \sup_{U\subseteq\Omega}\E{ \exp\left( 2C\cdot |\cle^\text{n-nested}_\kappa(U)|_{\mathcal{T}_\H} \right)}^{j/2+1},
\end{align*} 
which is finite due to Theorem \ref{thm:main} and Lemma \ref{lm:complexity}. In particular $\prob{|\Gamma_\Lambda|_{\mathcal{T}_\H}>n}$ decays super-exponentially by Markov's inequality. Then Corollary \ref{cor:complexity} follows by taking the sum of $|\Gamma_\Lambda|$ for all $\Lambda\subseteq\{\lambda_1,\ldots,\lambda_N\}$ containing at least two punctures..
\end{proof}

\section*{Acknowledgements}
Yijun Wan is grateful to Dmitry Chelkak for suggesting this topic and numerous helpful suggestions. Part of this work took place during a visit of Yijun Wan to Tsinghua University hosted by Hao Wu. The authors are grateful for her fruitful ideas during the early stage of preparing this paper. We thank Mikhail Basok for his careful reading and comments throughout the preparation of this paper. We also would like to thank Federico Camia, Nicolas Curien for helpful comments and remarks. Last but not least, we thank the anonymous reviewer for his/her careful reading of our manuscript and many insightful comments and suggestions.

\end{document}